\documentclass[twoside]{article}

\usepackage[accepted]{aistats2024preprint}
\usepackage{algorithm}
\usepackage{algpseudocode}
\usepackage{natbib}
\usepackage{empheq}
 \usepackage{color}
\usepackage[pagebackref=true,backref=true]{hyperref}
\hypersetup{
	colorlinks=true,        
	linkcolor=blue,         
	citecolor=blue,         
	urlcolor=blue           
}
\usepackage{cleveref}

\usepackage{booktabs}
\usepackage{threeparttable}

\usepackage{amsmath,amsthm,amssymb,amsfonts}
\usepackage{bm}

\newtheorem{theorem}{Theorem}
\newtheorem{assumption}{Assumption}
\newtheorem{definition}{Definition}
\newtheorem{lemma}{Lemma}
\newtheorem{proposition}{Proposition}
\newtheorem{remark}{Remark}
\newtheorem{corollary}{Corollary}
\crefname{assumption}{assumption}{assumptions}

\renewcommand{\gg}{\mathbf{g}}

\providecommand{\pp}{\mathbf{p}}
\providecommand{\qq}{\mathbf{q}}

\providecommand{\xx}{\mathbf{x}}
\providecommand{\yy}{\mathbf{y}}

\newcommand{\nk}{|\cI_k|}
\newcommand{\nl}{|\cI_\ell|}
\newcommand{\NK}{\sum_{k<K}\nk}

\newcommand{\brhom}{\bar\rho^{-1}}

\newcommand{\E}{\mathbb{E}}

\providecommand{\norm}[1]{\left\lVert#1\right\rVert}

\newcommand{\cI}{\mathcal{I}}

\newcommand{\cO}{\mathcal O}

\newcommand{\R}{\mathbb R}

\newcommand{\cE}{\mathcal{E}}
\renewcommand{\empty}{\varnothing}

\DeclareMathOperator{\argmin}{argmin}

\def\<#1,#2>{\langle #1,#2\rangle}

\usepackage{dsfont}

\renewcommand{\leq}{\leqslant}
\renewcommand{\geq}{\geqslant}

\renewcommand{\ge}{\geqslant}

\def\<{\langle}
\def\>{\rangle}
\def\|{\Vert}

\newcommand{\esp}[1]{\mathbb{E}\left[#1\right]}

\newcommand{\NRM}[1]{{{\left\| #1\right\|}}} 
\newcommand{\set}[1]{{{\left\{ #1\right\}}}} 
\newcommand{\proba}[1]{\mathbb{P}\left(#1\right)}

\renewcommand{\P}{\mathbb{P}}

\newcommand{\cV}{\mathcal{V}}

\newcommand{\cF}{\mathcal{F}}
\newcommand{\F}{\mathcal{F}}
\newcommand{\cP}{\mathcal{P}}

\newcommand{\N}{\mathbb{N}}
\newcommand{\cD}{\mathcal{D}}
\newcommand{\one}{\mathds{1}}

\newcommand{\myparagraph}[1]{\textbf{#1}}

\newcommand{\nex}{\mathrm{next}}
\newcommand{\prev}{\mathrm{prev}}

\usepackage{amsfonts}
\usepackage{amsthm}
\usepackage{amsmath}
\usepackage{amssymb}
\usepackage{mathrsfs}
\usepackage{amscd}
\usepackage{caption}
\usepackage{indentfirst}
\usepackage{cleveref}

\newcommand{\edgevkwk}{{\set{v_k,w_k}}}
\newcommand{\edgevlwl}{{\set{v_\ell,w_\ell}}}
\newcommand{\edgeuu}{{\{u,u'\}}}
\newcommand{\edgeuv}{{\{u,v\}}}
\newcommand{\edgevw}{{\{v,w\}}}
\newcommand{\edgeuw}{{\{u,w\}}}

\usepackage{xcolor}

\begin{document}

%
\runningtitle{Asynchronous SGD on Graphs}

%

\twocolumn[

\aistatstitle{Asynchronous SGD on Graphs: a Unified Framework for Asynchronous Decentralized and Federated Optimization}

\aistatsauthor{ Mathieu Even \And Anastasia Koloskova
 \And  Laurent Massoulié }

\aistatsaddress{ Inria - ENS Paris \And EPFL, Switzerland \And Inria - ENS Paris } ]

\begin{abstract}
    Decentralized and asynchronous communications are two popular techniques to speedup communication complexity of distributed machine learning, by respectively removing the dependency over a central orchestrator and the need for synchronization. Yet, combining these two techniques together still remains a challenge. 
    In this paper, we take a step in this direction and introduce Asynchronous SGD on Graphs (AGRAF SGD) --- a general algorithmic framework that covers asynchronous versions of many popular algorithms including SGD, Decentralized SGD, Local SGD, FedBuff, thanks to its relaxed communication and computation assumptions.
    We provide rates of convergence under much milder assumptions than previous decentralized asynchronous works, while still recovering or even improving over the best know results for all the algorithms covered.

\end{abstract}

\section{Introduction}

We consider solving stochastic optimization problems that are distributed amongst $n$ agents (indexed by a set $\cV$) who can compute stochastic gradients in parallel. This includes classical federated setups, such as distributed and federated learning. Depending on the application, agents have access to either same shared data distribution or a different agent-specific distributions.
In recent years, such stochastic optimization problems have continued to grow rapidly in size, both in terms of the dimension $d$ of the optimization variable---i.e., the number of model parameters in machine learning---and in terms of the quantity of data---i.e., the number of data samples $m$ being used over all agents. With $d$ and $m$ regularly reaching the hundreds or thousands of billions \citep{chowdhery2022palm,touvron2023llama}, it is increasingly necessary to use parallel optimization algorithms to handle the large scale.

With \emph{communication cost} being one of the major bottlenecks of parallel optimization algorithms, there are several directions aimed to improve communication efficiency. Amongst the others (such as local update steps \citep{stich2018local,woodworth2020local} and communication compression \citep{alistarh2017qsgd,pmlr-v97-koloskova19a_chocosgd}), \textbf{decentralization} and \textbf{asynchrony} are the two popular techniques for reducing the communication time. 
Decentralization \citep{koloskova2020unified,lian2017candecentralized} eliminates the dependency on the central server---frequently a major bottleneck in distributed learning---while naturally amplifying privacy guarantees \citep{cyffers2022muffliato}.
Asynchrony \cite{recht2011hogwild,baudet1978asynchronous,tsitsiklis1986distributed} shortens the time per computation rounds and allows more updates to be made during the same period of time.
It aims to overcome several possible sources of delays: nodes may have \emph{heterogeneous hardware} with different computational throughputs~\citep{kairouz2019advances,horvath2021fjord}, \emph{network latency} can slow the communication of gradients, and nodes may even just \emph{drop out}~\citep{ryabinin2021moshpit}.
Moreover, slower ``\textit{straggler}'' compute nodes can arise in many natural parallel settings, including training ML models using multiple GPUs~\citep{chen2016revisiting} or in the cloud; sensitivity to these stragglers poses a serious problem for synchronous algorithms, that depend on the slowest agent. In decentralized synchronous optimization where communication times between pairs of nodes may be heterogeneous, the algorithm can even be further slowed down by \textit{straggling communication links}.

Combining both decentralization and asynchrony is a challenging problem, and it is only recently that this question has risen a surge of interest \citep{assran2021asynchronouspush,bornstein2023swift,luo2020prague,liu2022asynchronous,nadiradze2021asynchronous,even2021delays,zhang2021fully}.
These works are however restricted to a given communication protocol and static topologies \citep{assran2021asynchronouspush,lian2015asynchronous,bornstein2023swift,nadiradze2021asynchronous,even2021delays}, no communication delays \citep{lian2015asynchronous,bornstein2023swift,nadiradze2021asynchronous}, or their analyses rely on an upper-bound on the maximal computation delay \citep{assran2021asynchronouspush,lian2017asynchronous,bornstein2023swift,luo2020prague,liu2022asynchronous,nadiradze2021asynchronous,zhang2021fully,pmlr-v202-wu23n}.
In this work we aim to circumvent these shortcomings. We study an asynchronous version of decentralized SGD in a unified framework that relaxes overly strong communication assumptions imposed by prior works. Our framework covers time-varying topologies, arbitrary computation orders and local update steps. We prove an improved rates of convergence under such a weaker communication assumptions, covering and improving asynchronous versions of many common distributed and federated algorithms.

\subsection{Contributions}

\textbf{(i)} We introduce \textbf{AGRAF SGD} (Asynchronous SGD on graphs), a unified formulation of an asynchronous version of the synchronous Decentralized SGD as formulated by \cite{koloskova2020unified}. One of the strengths of AGRAF SGD is that it formally takes the form of a simple sequence (\Cref{eq:agraf}), allowing for an effective theoretical analysis, while covering asynchronous versions of many distributed algorithms such as Asynchronous SGD, Decentralized SGD, FedAvg or FedBuff.

\textbf{(ii)} We analyze the AGRAF SGD sequence under various combinations of convexity, non-convexity, smoothness and Lispchitzness assumptions. We use a relaxed communication assumption that only imposes that the different topologies mix in a given window of time, while our computation assumption depends on whether the local functions are homogeneous or heterogeneous.
In special cases, our rates recover best known rates of Minibatch SGD, Asynchronous SGD or Decentralized SGD, while for Asynchronous Decentralized SGD, our rates improve the previous works by up to factors of order $n^2$, under relaxed assumptions (as summarized in \Cref{tb:rates}).

\textbf{(iii)} Finally, we show that AGRAF SGD allows to efficiently handle communication delays in decentralized optimization, by introducing \textit{Decentralized SGD on Loss Networks}. We show that the assumptions required in our analysis are satisfied by this algorithm, giving explicit rates of convergence that depend on the underlying graph topology, pairwise communication delays, and each device computation time.

\subsection{Related works}

\myparagraph{Asynchronous optimization.}
Asynchronous optimization has a long history. In the 1970s, \cite{baudet1978asynchronous} considered shared-memory asynchronous fixed-point iterations, and an early convergence result for Asynchronous SGD was established by \cite{tsitsiklis1986distributed}. Recent analysis typically relies on bounded delays \citep{agarwal2011distributed,recht2011hogwild,lian2015asynchronous,stich2020error}, while some algorithms try to adapt to the delays \citep{sra2016adadelay,zheng2017asynchronous,mishchenko2018delay,koloskova2022sharper,asynchronous_sgd_mischenko,Feyzmahdavian2021asynchronous}, in order to depend only on an average delay. For more examples of stochastic asynchronous algorithms, we refer readers to the surveys by \cite{ben2019demystifying,assran2020advances}. 
More closely related to our analysis techniques, \cite{mania2017perturbed} proposed and utilized the analysis tool of \textbf{virtual iterates} for Asynchronous SGD under bounded delays, extended by \cite{koloskova2020unified,asynchronous_sgd_mischenko} who proved that Asynchronous SGD performs well under arbitrary delays. We adapt this proof approach to decentralized optimization in order to obtain some robustness towards large delays and introduce a different virtual sequence for the averaged model over all the nodes.

\myparagraph{Decentralized SGD and asynchrony.}
Decentralized SGD \citep[e.g.]{koloskova2022sharper} consists in iterations where at every time step, all nodes perform local SGD steps, and communicate their local model with their neighbors in a graph (that may vary with time, but that needs to mix in an ergodic way).
The closest works to ours \cite{lian2017asynchronous,bornstein2023swift} proposed asynchronous versions of decentralized SGD where at each iteration, \textbf{one} node $v_k$ is sampled \textit{independently from the past} (with fixed probabilities), and this node performs a local stochastic gradient step and an averaging operation with its neighbors. We extend their sequence and results to a more general (due to relaxed communication and computation assumptions)
asynchronous version of decentralized SGD, that keeps the ‘‘unified'' point of view of the work of \cite{koloskova2020unified}.
\cite{assran2021asynchronouspush} considers asymmetric communications (\textit{push sum}) and all the agents performing computations at every iterations in a synchronous way, \cite{nadiradze2021asynchronous} considers quantized pairwise communications as in the historical gossip algorithm \citep{boyd2006gossip}, but no communication nor computation delays, while \cite{luo2020prague,agarwal2009information} do not provide convergence guarantees. Orthogonally, \cite{even2021delays} consider both communication and computation delays in a \textit{continuized} framework \citep{even2021continuized}, allowing more degrees of freedom for the algorithm and the analysis, but their work does not apply to modern ML tasks; still, our Loss Network section relates to this line of work due to the introduction of continuous-time physical delays.

\section{AGRAF Algorithmic Framework}
In this section we present AGRAF SGD---our algorithmic framework for asynchronous decentralized SGD---and give examples of existing popular algorithms that it can cover.
\subsection{Asychronous SGD on graphs}

We consider a connected undirected graph $G=(\cV,\cE)$\footnote{Since we consider varying topologies, this graph should be thought as the union of graphs considered over time.} on a set of nodes $\cV=\set{1,\ldots,n}$. 
Let the function $f_v:\R^d\to\R$ of agent $v \in \cV$ be defined as
\begin{equation}\label{eq:local_obj}
    f_v( x):=\esp{F_v( x,\xi_v)}\quad \xi_v\sim \cD_v\,,\quad  x\in\R^d\,,
\end{equation}
where $\cD_v$ is some local distribution.
Let the global objective function $f:\R^d\to\R$ be defined as follows, and consider the optimization problem \vspace{-5pt}
\begin{equation}
    \label{eq:the-objective}
    \min_{x\in\R^d}\set{f(x):=\sum_{v\in\cV}q_v f_v(x)}\,,\vspace{-5pt}
\end{equation}
for some non-negative weights $(q_v)$ that sum to 1.
We classically assume that node $v$ in the graph has access to unbiased stochastic gradients of $f_v$ (of the form $F_v( x,\xi_v)$).
The standard goal of decentralized optimization is to minimize $f$ using only local computations and communications (only neighboring nodes in the graph can communicate).

\myparagraph{Notations.}
Standard small letters ($x,g,y,z$, etc) are for vectors in $\R^d$. Capital letters (mostly $W$) are for matrices in $\R^{\cV\times \cV}$. Bold letters $\xx,\gg,\ldots$ are for \emph{concatenated vectors} in $\R^{\cV\times d}$, that we write as $\xx=(x_v)_{v\in\cV}$. For some vector $x\in\R^d$, we denote $\xx\in\R^{\cV\times d}$ the concatenated vector such that $\xx_v=x$ for all $v\in\cV$.
$\one\in\R^\cV$ is the vector with all entries equal to 1.
For $\xx\in\R^{\cV\times d}$, we denote $\bar \xx=\frac{1}{n}\one\one^\top \xx$.

In this paper we study a general scheme for \textbf{\emph{asychronous SGD on graphs (AGRAF)}} which is summarized in \Cref{algo:AGRAF}: workers asynchronously perform local SGD steps (lines 3-4), while an underlying \textbf{linear communication algorithm} is running \textit{without incurring communication delays} (line 7). A linear communication algorithm on graph $G$ implies that any communication update can be formulated as $\xx_+=W\xx_-$ where $\xx_+,\xx_-\in \R^{\cV\times d}$ are respectively the global state after and before the communication update, and $W\in\R^{\cV\times \cV}$ is a \textbf{communication matrix} with $W_{v,w}$ being zero for disconnected nodes $v, w$, i.e. $W_{v,w}\ne 0$ iff $\edgevw\in\cE$.

\begin{algorithm}[t]
    \caption{Asynchronous SGD on graph $G$ (AGRAF SGD)}
    \label{algo:AGRAF}
    \begin{algorithmic}[1]
    \State \textbf{Input:} 
    $\bar x^0\in \R^{d}$, $ x_v=\bar  x^0$ for $v\in\cV$ initialized local variables, stepsize $\gamma>0$
    \For{$v\in\cV$,}
        \State Upon finishing computation of a stochastic gradient $\nabla F(\tilde x_v,\tilde\xi_v)$ at some previous local current state $\tilde x_v$,
        \begin{equation*}
             x_v\longleftarrow  x_v-\gamma \nabla F_v(\tilde x_v,\tilde\xi_v)\,.
        \end{equation*}
        \State Compute $\nabla F_v( x_v,\xi_v)$ for $\xi_v\sim\cD_v$ independently from the past, at current state $ x_v$.
     \EndFor
    \While{procedure still running}
        \State Run any \textbf{linear communication algorithm} on graph $G$ incurring no communication delay.
    \EndWhile
    \end{algorithmic}	
\end{algorithm}

Since every agent asynchronously works at their own speed and communicates in a decentralized way, there is no global state. Keeping track of a global ordering of the iterates involving both computation and communication updates is thus a challenge. In the next subsection we address this challenge and propose a way to effectively cast Algorithm~\ref{algo:AGRAF} into equations with ordered updates. This reformulation is a key novelty of our work. It allows for an improved theoretical analysis with better rates together with relaxed communication and computations assumptions, allowing AGRAF SGD to cover asynchronous versions of many popular distributed and federated algorithms.

\subsection{The sequence studied}

We denote by $T_0=0$ the initialization time of the algorithm  and by $\set{0<T_1<T_2<\ldots}$ the times at which the local computation updates are made. 
Note that these are physical (continuous) times, and that several agents may possibly finish their local computations at the same time $T_k$. We also assume that computational updates are \textbf{atomic}.
For some time $T$, we denote as $T-$ the left limit ($\lim_{t\to T,t<T}$) and $T+$ the right limit ($\lim_{t\to T,t>T}$).
For time $t\in\R^+$ (physical time), let $ x_v(t)\in\R^d$ denote the state of the local variable at time $t$, and let $ \xx(t)=( x_v(t))_{v\in\cV}$.
For $k\geq 0$ and $v\in\cV$, let $ x^k_v$ denote the state of the local variable at node $v$ at time $T_k+$ \emph{i.e.}, $ x_v^k= x_v(T_k+)=\lim_{t\to T_k,t>T_k} x_v(t)$ and let $\xx^k=(x_v^k)_{v\in\cV}$.

\myparagraph{Communication updates.}
For $k\geq0$, none to plenty of communication updates may have happened between the computational update times $T_k$ and $T_{k+1}$.
We encode these communication updates by a \emph{single} matrix $W_k$: $W_k$ is thus the product of all communication matrices corresponding to communication updates between times $T_k$ and $T_{k+1}$.
Hence, we can write:
\begin{equation*}
     \xx({T_{k+1}-})=W_k  \xx({T_k+})\,.
\end{equation*}
If no communication happened between two gradients computed, we have $W_k = I_d$. If there are $r$ communications between times $T_k+$ and $T_{k+1}-$ that happened at times $T_k<T_{k,1}<\ldots<T_{k,r}<T_{k+1}$, denoting by $W_{k,r}$ the communication matrix corresponding to communication updates at time $T_{k,r}$, we have $W_k=W_{k,r}\cdot\ldots \cdot W_{k,2}\cdot W_{k,1}$. Note that for $r=0$ this product is taken equal to $I_d$.

\myparagraph{Computation updates.} For $k\geq1$, let $\cI_k\subset\cV$ be the set of nodes that finish computing stochastic gradients $\nabla F_v(\tilde x_{v}^k,\tilde\xi_{v}^k)$ for $v\in\cI_k$ at time $T_k-$.
The computation updates, that are assumed to be \textbf{atomic}, then read:
\begin{equation*}
     x_{v}({T_k+})= x_{v}({T_k-})-\gamma \nabla F(\tilde x_{v}^k,\tilde\xi_{v}^k)\,,\quad v\in \cI_k\,,
\end{equation*}
where $\tilde x_{v}^k= x_{v}^{k-1-\tau(k,v)}$ and $\tilde\xi_{v_k}^k=\xi_{v_k}^{k-1-\tau(k,v)}$, for $\tau(k,v)\geq 0$ the delay of this update that corresponds to the number of computation updates performed by other nodes during the computation of the local stochastic gradient.

\myparagraph{The sequence studied.} Combining communication and computation updates, the sequence generated by Algorithm~\ref{algo:AGRAF} follows the following recursion:\vspace{-3pt}
\begin{equation}\label{eq:agraf}
     \xx^{k+1}=W_k \xx^k - \gamma   \gg^{k}\,,\vspace{-3pt}
\end{equation}
where $  g^{k}_w=0$ for $v\notin \cI_k$, and $g^{k}_{v}=\nabla F_v( x_{v}^{k-\tau(k,v)},\xi_{v}^{k-\tau(k,v)})$.

What is important to keep in mind is that the iterates $\xx^{k + 1}$ are taken at the time just after computation updates (time $T_k+$) so that $k$ denotes the number of computation updates. $\cI_k$ is the set of nodes that perform computation updates at iteration $k$, it can be any subset of $\cV$, and $\NK$ denotes the total number of stochastic gradients computed up to iteration $K$ by \textbf{all} the agents.
The matrix $W_k$ encodes all communications that happened between the $k$-th and $(k+1)$-th computation updates (there can be any number such communications, the more there are the more $(W_k)$ will mix).

\subsection{AGRAF SGD is the right formulation of Asynchronous Decentralized SGD}\label{sec:agraf_sgd_right_form}

Recall that the Decentralized SGD algorithm \citep[e.g.]{koloskova2022sharper} consists in iterations of the form:\vspace{-3pt}
\begin{equation}\label{eq:dsgd}
    x_v^{k+1}=\sum_{w\sim v} W_\edgevw^{(k)} x_w^k - \nabla F_v(x^k_v,\xi_v^k)\,,\quad \forall v\in\cV\,,\vspace{-3pt}
\end{equation}
for communication matrices $(W^{(k)})_k$ satisfying \Cref{hyp:consensus}. The question thus arises: how can Decentralized SGD be turned into an asynchronous algorithm?
Previous works \citep{lian2017asynchronous,bornstein2023swift} proposed and analyzed schemes that take the following form: at each iteration, \textbf{one} node $v_k$ is sampled independently from the past (with fixed or lower bounded probability), and this node performs a local stochastic gradient step together with an averaging operation with its neighbors in the graph.
This results in updates of the form of AGRAF SGD, for $\cI_k=\set{v_k}$ and $W_k$ a matrix that depends on $v_k$ and that mixes (in mean) independently from the past ($\esp{W_k|W_0,\ldots,W_{k-1}}$ mixes well).

Leaving the analyses aside, this prior approach it too restrictive: \textbf{(i)} \emph{communication assumptions} do not allow varying topologies that may mix but only in the long run, which may particularly be the case for asynchronous algorithms, and \textbf{(ii)} \emph{computation assumptions} do not allow for more than one worker to update their value at the same time; having a sampling assumption restricts the type of delays that the algorithm can handle; and nodes that compute should not necessarily be correlated to communicating edges since this forbids the use of several local SGD steps.

AGRAF SGD thus appears as a natural way to make Decentralized SGD asynchronous: nodes are not forced to all perform computations at the same time as in \cref{eq:dsgd}, and having the relaxed communication assumption (\Cref{hyp:consensus}) allows any communication order, especially when one considers $W_k$ as a concatenation of all communications that may happen between two consecutive computations.

\subsection{Some examples covered by AGRAF}

We now give a few examples of algorithms (\emph{i.e.} communication and computation schedules) that can be cast as AGRAF SGD.
The three first are degenerate cases.

\myparagraph{Minibatch SGD} and \textbf{Asynchronous SGD} are obtained by setting $W_k=\frac{1}{n}\one\one^\top$, and $\cI_k=\cV$ and $\cI_k=\set{v_k}$ for some node $v_k$ respectively.

\myparagraph{Decentralized (local) SGD.}
Set $\cI_k=\cV$ and $(W_k)_k$ a sequence of gossip matrices to obtain Decentralized SGD \citep{SundharRam2010}. Note that in that case there are no computation delays, since this algorithm is inherently synchronous and all nodes perform updates at the same time. As done in \cite{koloskova2020unified}, periodic communications are possible, allowing to recover algorithms with several local gradient steps between each communication round, such as \textbf{Local (Decentralized) SGD} \citep{stich2018local} or \textbf{FedAvg} \citep{mcmahan2017communication}.

\myparagraph{Asynchronous Decentralized SGD.}
As explained in \Cref{sec:agraf_sgd_right_form}, AGRAF SGD covers Asynchronous Decentralized SGD beyond particular instances previously studied \citep{lian2017asynchronous,bornstein2023swift}.
Furthermore, since we make relaxed communication/computation assumptions, we cover more general decentralized algorithms that allow \textbf{local} gradient steps between communications, varying topologies, and arbitrary computations. As such, together with covering an asynchronous version of Decentralized SGD \eqref{eq:dsgd}, we also cover asynchronous versions of \textbf{FedAvg} or\textbf{Local SGD}, together with \textbf{FedBuff} \citep{pmlr-v151-nguyen22b_FedBuff}.

\myparagraph{Asynchronous SGD on Loss Networks.}
If communication latencies are not negligible compared to computational ones, designing an algorithm that is asynchronous and decentralized becomes much more challenging, as the naive implementation might lead to deadlocks. 
In order to handle non-negligible communication delays,  we use \textbf{\emph{loss networks}} \citep{lossnetworks1991kelly} to enforce that the edges adjacent to \textit{‘‘busy''}  nodes are prohibited to be used for communicating\footnote{Loss Networks were initially introduced by F. Kelly to model telecommunication networks, where the same mobile phone cannot initiate another phone call while being busy with another call. In our case, phone calls should be thought as communicating with a neighbor.}.
This enables us to design communication/computation schemes that fit in the AGRAF framework, while not violating the physical delay constrains.
We introduce these Loss Networks in \Cref{sec:LN}: we define them more thoroughly, and provide their ergodic mixing properties with explicit constants that depend on the graph topology and local communication and computation delays.


\section{Assumptions and Notations}

We consider solving the problem \eqref{eq:the-objective} under several standard \cite[see, e.g.,][]{convexoptim_bubeck} combinations of conditions on the objective $F$.
We denote the minimum of~$f$ as $f^* \coloneqq \min_{x\in\R^d} f(x)$, an upper bound on the \textbf{initial suboptimality} $\Delta \geq f(\bar x^0) - f^*$, and an upper bound on the \textbf{initial distance} to the minimizer $D \geq \min\set{\NRM{x_0-x^*} : x^* \in \argmin_x f(x)}$ that we assume to exist.
$\NRM{\cdot}$ denotes the Euclidean norm.
A function $F$ is \textbf{convex} if for each $x,y$ and subgradient $g \in \partial F(x)$, we have $F(y) \geq F(x) + \langle g,y-x\rangle$.
When $F_v$ and $f_v$ are convex, we do not necessarily assume they are differentiable, but we abuse notation and use $\nabla f_v(x)$ and $\nabla F_v(x;\xi)$ to denote an arbitrary subgradient at $x$.
The loss $F_v$ is $B$-\textbf{Lipschitz-continuous} if for each $x,y$ and $\xi$, we have $|F_v(x;\xi) - F_v(y;\xi)| \leq B\NRM{x - y}$.
The objective $f_v$ is $L$-\textbf{smooth} if it is differentiable and its gradient is $L$-Lipschitz-continuous. 
We also assume the stochastic gradients have $\sigma^2$-\textbf{bounded variance}\footnote{which can easily be generalized to $\esp{\NRM{\nabla f_v(x)-\nabla_xF_v(x,\xi_v)}^2}\leq \sigma^2 + \delta^2 \NRM{\nabla f_v(x)}^2$.}.

\begin{assumption}[Noise]\label{hyp:noise_variance}
    There exists $\sigma^2$ such that for all $x$ and $v\in\cV$, we have $\esp{\nabla_xF_v(x,\xi_v)}=\nabla f_v(x)$ and $\esp{\NRM{\nabla f_v(x)-\nabla_xF_v(x,\xi_v)}^2}\leq \sigma^2$, where $\xi_v\sim \cD_v$.
\end{assumption}


\myparagraph{Graph, communications and mixing.}
We now formulate the communication assumptions we will make. For $k\geq 0$, as opposed to some previous Asynchronous Decentralized SGD analyses \citep{lian2017asynchronous,bornstein2023swift}, we do not want to assume that $W_k$ mixes well in mean (i.e., that the spectral gap of $\esp{W_k|W_0,\ldots,W_{k-1}}$ is non-null or some other related assumption), since $W_k$ may possibly be the identity matrix.
We use the least restrictive assumption under which convergence of (synchronous) decentralized SGD is established \citep{koloskova2020unified}, by assuming that if we wait enough communication updates, a consensus will ultimately be achieved.

\begin{assumption}[Ergodic mixing]\label{hyp:consensus}
$W_k\one=\one$ and  there exist $\rho,k_\rho>0$ such that we have $\forall k\in\N$ and $\forall  \xx\in\R^\cV$:
\begin{equation}\label{eq:hyp_consensus}
\begin{aligned}
     &\esp{\NRM{W^{(k:k+k_\rho)} \xx- \bar\xx}^2\!|\cF_k}\!\leq\! (1-\rho)^2\NRM{ \xx-\bar \xx}^2,
\end{aligned}
\end{equation}
where for $k,\ell\geq0$, $W^{(k,\ell)}=W_{\ell-1}\ldots W_{k+1}W_k$, and $\cF_k=\sigma(\xx^s,\gg^{s-1},W_{s-1},s\leq k)$ is the filtration up to step $k$.
\end{assumption}

This assumption makes it possible to consider any ‘‘reasonable'' communication scheme. 
In the rest of the paper, when assuming that \Cref{hyp:consensus} holds for some constants $(\rho,k_\rho)$, we write $\bar\rho=\frac{e-1}{e}\frac{\rho}{k_\rho}$ (with $e=\exp(1)$), and this quantity is used in our main results.

\myparagraph{Heterogeneous and homogeneous settings, sampling assumptions.}
Assuming that the sequence of nodes $(\cI_k)_{k\geq0}$ that iteratively perform local updates is arbitrary makes it possible to encompass all possible computation orderings and cover arbitrary delays. It is much more general than assuming that $\cI_k=\cV$ for all $k$ (decentralized SGD) or $\cI_k=\set{v_k}$ for $v_k$ sampled independently from the past, as assumed in most previous asynchronous decentralized works \citep{lian2015asynchronous,bornstein2023swift}.
However, if functions $f_v$ are not all equal and if the sequence $v_k$ is arbitrary, convergence to the global function $f$ cannot be assured (some of the nodes $v$ might simply never appear during training). We therefore need to make some sampling assumption if we assume that local functions can be heterogeneous. We will thus assume either one the two following assuptions: \textbf{(i)} the heterogeneous setting where local functions $f_v$ can be different, but where we make some node-sampling assumption for computations, and \textbf{(ii)} the homogeneous setting, where computations can be arbitrary, but functions $f_v$ are all the same.
Note that it is classical in asynchronous optimization to either assume \textbf{(i)} or \textbf{(ii)}; for instance, Asynchronous SGD with arbitrary orderings is proved to converge only under such assumptions \citep{asynchronous_sgd_mischenko,koloskova2022sharper}. However, Asynchronous Decentralized works only assume that the sampling assumption \textbf{(i)} holds.
Formally, we summarize these into the following two assumptions.
\begin{assumption}[Heterogeneous setting]\label{hyp:hetero}
    There exists $\zeta^2$ such that the \textbf{population variance} satisfies: \vspace{-5pt}
    \begin{equation}\label{eq:pop_variance}
      \sum_{v\in\cV}q_v\NRM{ \nabla f_v(x)-\nabla f(x)}^2\leq \zeta^2\,, \quad \forall x\in\R^d\,.  \vspace{-5pt}
    \end{equation}
    There exists $\pp=(p_v)_{v\in\cV}\in[0,1]^\cV$ such that the sequence $(\one_{v\in\cI_k})_{k\geq 0}$ is \emph{i.i.d.}~distributed, with $\proba{v\in\cI_k}=p_v$ for all $k\geq0,v\in\cV$.
    We denote $\kappa_{\pp}=\frac{p_{\max}}{\bar p}$, $p_{\max}=\max_v p_v$ and $\bar p=\sum_{v\in\cV}p_v$
    Furthermore, we assume that $\pp$ is proportional to $\qq$: $\pp=\beta \qq$, and since $\sum_v q_v =1$, we thus have $\beta= n \bar p$.
\end{assumption}

\begin{assumption}[Homogeneous setting]\label{hyp:homo}
    All functions $f_v$ satisfy $f_v\equiv f$. No assumption on $(\cI_k)_{k\geq0}$.
\end{assumption}

\section{General Convergence Analysis}

We now turn to our main results: convergence guarantees for AGRAF SGD, under a variety of regularity assumptions and settings.
Note that in almost all cases, our rates do not depend on any upper bound on the maximal delays, which is a key feature of our analysis. This is also the case for asynchronous SGD \citep{koloskova2022sharper,asynchronous_sgd_mischenko} or a recent asynchronous decentralized SGD work \citep{bornstein2023swift}.
In this section, while presenting the results, we will only compare our results to degenerate baselines such as minibatch SGD, asynchronous SGD or decentralized SGD, in order to give simple arguments to show that our rates have expected order of magnitudes, leaving more complex comparisons and applications to be developed in \Cref{sec:applications}.
We first start with convex-Lipschitz losses. In this section, all the rates are obtained for \textbf{a constant stepsize $\gamma$} (that differs in each different case and is time-horizon dependent), explicited in the proofs in the Appendix. 

\begin{theorem}[Lipschitz-convex rate]\label{thm:lip-conv}
    Assume that $f$ is convex and that for almost all (\textit{i.e.}, with probability 1) $\xi\sim\cD_v$ $F_v(\cdot,\xi)$ is $B$-Lipschitz for some $B>0$, let $D^2\geq \NRM{ x_0- x^\star}^2$, and $F_K=\esp{f\left(\frac{1}{\NK}\sum_{k=0}^{K-1} \sum_{v\in\cI_k} x_{v}^k\right)-f( x^\star)} $.

            \textbf{1.} In the \textbf{homogeneous} setting (\Cref{hyp:homo}), 
        \begin{align*}
            F_K=\cO\left(\sqrt{\frac{B^2D^2n\bar\rho^{-1}}{\NK}}\right)\,.
        \end{align*}
        
        \textbf{2.} In the \textbf{heterogeneous} setting (\Cref{hyp:hetero}),
        \begin{align*}
            F_K=\cO\left(\sqrt{\frac{B^2D^2}{\NK}\times n\sqrt{p_{\max}}(\sqrt{\kappa_\pp}+\bar\rho^{-1})}\right)\,.
        \end{align*}

\end{theorem}

We thus recover the well-known rate of minibatch SGD for convex-Lipschitz losses, by setting $\bar\rho=1$ and $\nk=n$, leading to the optimal rate $\cO(\sqrt{B^2D^2/K})$ \citep{nemirovskyyudin1983}.
Asynchronous SGD has also been studied under such assumptions, with the rate $\cO(\sqrt{B^2D^2n/K})$ that we recover here ($\bar\rho=1$ and $\nk=1$) \citep{asynchronous_sgd_mischenko}, that is minmax optimal \citep{woodworth2018graph}.
No rates under the given assumptions existed for Decentralized (local) SGD, that thus exhibits a rate of $\cO(\sqrt{B^2D^2\brhom/K})$.
Finally, adding the sampling assumption not only enables to handle heterogeneous functions, but also leads to improved rates: for well balanced weights ($p_v\approx \bar p$ and $\kappa_\pp\approx 1$) we have $n\sqrt{\bar p}\brhom$ instead of $n\brhom$, which can improve the rate by a factor $1/\sqrt{n}$ if $\cO(1)$ agents compute at the same time, which is usually the case in the asynchronous setting. This phenomenon (better rates under the sampling assumption) appears in all our other rates below.

\begin{theorem}[Lipschitz-smooth-convex rate]\label{thm:lip-smooth-conv}
    Assume that $f$ is convex, for almost all $\xi\sim\cD$, $F(\cdot,\xi)$ is $B$-Lipschitz for some $B>0$, $f_v$ is $L$-smooth, \Cref{hyp:noise_variance} holds, and let $D^2\geq \NRM{ x_0- x^\star}^2$.
        In the \textbf{homogeneous} setting, 
        \begin{align*}
            &\esp{f\left(\frac{1}{\NK}\sum_{k=0}^{K-1} \sum_{v\in\cI_k} x_{v}^k\right)-f( x^\star)}\\
            &=\cO\left( \frac{L\bar\rho^{-1} nD^2}{ \NK}\sqrt{\frac{\sigma^2B^2}{\NK}}\right.\\
            &\quad+ \left.\left(\!\frac{D^2n\sqrt{L\left(B^2+\bar\rho^{-1}\sigma^2\right)}}{\NK}\right)^{\frac{2}{3}}\!\!\right)
        \end{align*}
\end{theorem}

For Lipschitz-smooth functions, setting $\brhom=1$ and $\nk=1$, we recover the exact same rates as Asynchronous SGD under arbitrary delays, recently derived by \cite{asynchronous_sgd_mischenko,koloskova2022sharper}, and that do not depend on any upper bound on the delays.
These rates are thus extended to the more general AGRAF SGD algorithm.

\begin{theorem}[Smooth-convex]\label{thm:smooth-conv}
    Assume that $f$ is convex, all $f_v$ are $L$-smooth, and let $D^2\geq \NRM{ x_0- x^\star}^2$.
    
         \textbf{1.} In the \textbf{homogeneous} setting, 
        \begin{align*}
            &\esp{f\left(\frac{1}{\NK}\sum_{k=0}^{K-1}\sum_{v\in\cI_k} x_{v}^k\right)-f( x^\star)}\\
            & =\cO \left( \frac{LD^2 (n\brhom+\sqrt{n\tau_{\max}})}{\NK} + \sqrt{\frac{D\sigma^2}{\NK}}\right.\\
            &\quad\left.+ \left[  \frac{D^2\sqrt{L\sigma^2n^2\brhom }}{\NK} \right]^{2/3} \right)\,,\vspace{-5pt}
        \end{align*}
        where $\tau_{\max}\geq \sup_{k<K,v\in\cV} \sum_{\ell=k}^{\tau(k+1,v)} \nl$ is an upper bound on the maximal compute delay.
        
        \textbf{2.} In the \textbf{heterogeneous} setting, 
        \begin{align*}
            &\esp{f\left(\frac{1}{K}\sum_{k<K} \bar x^k\right) -f(x^\star)}\\
            & = \cO\left( \frac{LD^2\sqrt{\kappa_\pp}\left(\frac{1}{\bar p}  +  (\bar\rho\sqrt{\bar p})^{-1}  \right)}{K}  + \sqrt{\frac{D^2(\sigma^2+\zeta^2)}{n\bar p K}} \right.\\
            &\left.+ \left[ \frac{D^2\sqrt{L\sigma^2p_{\max}\bar\rho^{-1} + L\zeta p_{\max}\bar\rho^{-2}}}{\bar p K}  \right]^{\frac{2}{3}}   \right) \,.\vspace{-5pt}
        \end{align*}
%
\end{theorem}
Removing the Lipschitz assumption, we are still able to recover and extend the rates of Asynchronous SGD with \textbf{constant} stepsizes. Note that under no sampling assumption, this rate depends on $\sqrt{n\tau_{\max}}$ instead of $n$ as in the previous two theorems; however, this dependency is still better than depending on $\tau_{\max}$ since we always have $\tau_{\max}\geq n$. We expect to be able to remove this dependency by the use of varying stepsizes as was done for Asynchronous SGD (where stepsizes scale as $1/(L\tau(k))$, inversely proportional to the actual delay). However, such stepsizes cannot be used in a fully decentralized setting, since a given node cannot be aware of the iteration counter $k$ and thus of the delay $\tau(k)$.
Note also that in the sampling case, we have $\esp{\NK}= n\bar p K$, so that the statistical rate is still reached.
These comments also applies to the \textbf{non-convex and smooth setting} below,
for which we fall back to showing that the algorithm will find an approximate first-order stationary point of the objective.
We recover, as in the convex-smooth case just above, the exact same rates as \cite{koloskova2020unified} for Decentralized (local) SGD.

\begin{theorem}[Non-convex and smooth rates]\label{thm:smooth-nonconv}
    Assume that the functions $f_v$ are $L$-smooth.
    
        \textbf{1.}  In the \textbf{homogeneous} setting, 
        \begin{align*}
            &\esp{\frac{1}{\NK}\sum_{k < K} \nk \norm{\nabla f(\bar x^k)}^2}\\ 
            &=\cO\left( \frac{L F_0(\sqrt{n \tau_{\max}}\!+\! n \bar\rho^{-1})}{K}+\left(\frac{L \sigma^2 F_0}{K} \right)^{\frac{1}{2}}\right.\\
            &\quad \left.+ \left(\frac{L \sigma n F_0}{K \sqrt{\bar \rho}}\right)^{\frac{2}{3}} \right)\,.
        \end{align*}
        \textbf{2. }In the \textbf{heterogeneous} setting, 
        \begin{align*}
        &\esp{\frac{1}{K}\sum_{k < K}  \norm{\nabla f(\bar x^k)}^2}\\
            &=\cO\left( \frac{L F_0\sqrt{\kappa_\pp}(\frac{1}{\bar p}+ (\bar\rho\sqrt{\bar p})^{-1})}{K}+\left(\frac{L (\sigma^2+\zeta^2) F_0}{K} \right)^{\frac{1}{2}} \right.\\
            &\left. + \left(\frac{L n F_0\sqrt{\sigma^2p_{\max}\brhom + \zeta^2p_{\max}\bar\rho^{-2}}}{ K}\right)^{\frac{2}{3}} \right)\,.\vspace{-5pt}
        \end{align*}

    
\end{theorem}

\begin{remark}[Heterogeneous without sampling]
    So far, the heterogeneous setting was only considered under a sampling assumption. In fact, generalizing \cite[Theorem 4]{asynchronous_sgd_mischenko} to AGRAF SGD, under both \textbf{heterogeneous functions} with population variance $\zeta^2$ (as in \cref{eq:pop_variance}) and \textbf{arbitrary ordering of the updates}, the exact same rate as \Cref{thm:smooth-nonconv}.1 up to an additional term $\cO(\zeta^2)$ could be obtained.
\end{remark}

\section{Applications}\label{sec:applications}

\subsection{Better rates for Asynchronous Decentralized SGD}

\begin{table*}[t]
\caption{
\footnotesize We compare the number of iterations required to reach the statistical regime $\cO(\sigma^2/\NK)$ (if it is reached) of previous Asynchronous Decentralized SGD works \citep{lian2017asynchronous,bornstein2023swift} with our rates (\Cref{thm:lip-smooth-conv,thm:smooth-conv}). 
\textbf{Strong} communication assumption : \Cref{hyp:consensus} with $k_\rho=1$ and $W_k$ independent from the past; \textbf{Sampling assumption} : $\cI_k=\set{v_k}$ with $v_k$ \textit{i.i.d.} sampled \textit{or} $\P(v\in\cI_k)=p_v$ \textit{i.i.d.} sampled.
$^{{\color{blue}(a)}}$ \cite{bornstein2023swift} reaches $\cO({\color{red}\bar\rho^{-2}}\sqrt{\sigma^2/K})$ instead, after $\cO(n^2\bar\rho^{-4})$ iterations.
}\label{tb:rates}

\vspace*{.15cm}
	
\centering
\begin{threeparttable}
\resizebox{0.95\linewidth}{!}{
\begin{tabular}{c c c c c}
	\toprule[.1em]
	\textbf{Reference}  & \begin{tabular}{c}\footnotesize \textbf{Communication} \\ \footnotesize\textbf{Assumption}\end{tabular} & \begin{tabular}{c}\footnotesize\textbf{Computation}\\ \footnotesize\textbf{Assumption} \end{tabular} & \textbf{Regularity} & \begin{tabular}{c}\textbf{\footnotesize$\#$ iterations before}\\ \footnotesize\textbf{$\sqrt{\frac{\sigma^2}{K}}$ regime} \end{tabular}  \\
	\midrule
	\begin{tabular}{c}\cite{lian2017asynchronous}   \end{tabular} &  \footnotesize \textbf{Strong} & \footnotesize Sampling & \footnotesize Smoothness & $\cO(\max(n^4\bar\rho^{-4},\tau_{\max}^4))$ \\ \cmidrule{1-1}
	\begin{tabular}{c}\cite{bornstein2023swift}  \end{tabular} &  \footnotesize \textbf{Strong}& \footnotesize Sampling & \footnotesize Smoothness & {\color{red} N.A.}$^{{\color{blue}(a)}}$ \\ \cmidrule{1-1}
	\Cref{thm:lip-smooth-conv}.1 & \begin{tabular}{c} \Cref{hyp:consensus}\end{tabular} & \textbf{Arbitrary} & \footnotesize Smooth-\textbf{Lipschitz}, \footnotesize\textbf{Homogeneous}  & $\cO(n^4\bar\rho^{-2})$ \\ \cmidrule{1-1}
	\Cref{thm:smooth-conv}.1 & \begin{tabular}{c} \Cref{hyp:consensus}\end{tabular} & \textbf{Arbitrary} & \footnotesize Smooth, \textbf{Homogeneous} & $\cO(\max(n^4\bar\rho^{-2},n\tau_{\max}))$  \\ \cmidrule{1-1}
	\Cref{thm:smooth-conv}.2 & \begin{tabular}{c} \Cref{hyp:consensus}  \end{tabular} & \footnotesize Sampling & \footnotesize Smoothness & $\cO(n^2\bar\rho^{-4})$ 
	\\
	\bottomrule[.1em]
\end{tabular} }
\end{threeparttable}
\end{table*}

A first direct application of our theory is a better analysis of Asynchronous Decentralized SGD.
Comparing our analysis with those of \cite{lian2017asynchronous,bornstein2023swift}, we highlight that our work handles arbitrary computation orders and delays in the homogeneous settings, as opposed to \cite{lian2017asynchronous,bornstein2023swift} that are only valid for $k_\rho=1$ in \Cref{hyp:consensus} (which means that at any step $k$, conditionally on the current state, the graph of edges that can be sampled must be connected) and under a sampling assumption.
In both homogeneous and heterogeneous cases, our communication assumptions are much less restrictive. Furthermore, under similar computation and regularity assumptions as \cite{lian2017asynchronous,bornstein2023swift} (sampling and smooth losses, see last line of \Cref{tb:rates}), our convergence bound (\Cref{thm:smooth-conv}.2) reaches a statistical rate $\sqrt{\sigma^2/\NK}$ after $\NK=\cO(n^2\bar\rho^{-4})$, while \cite{bornstein2023swift} does not reach such a statistical rate and \cite{lian2017asynchronous} reaches this rate after $K=\cO(\max(n^4\bar\rho^{-4},\tau_{\max}^4))$ iterations. For the sake of comparison, we take $\bar p$ of order $1$ in our rates.

\subsection{Asynchronous Decentralized SGD on Loss Networks}\label{sec:LN}

The previous considerations and the AGRAF SGD rates hold as long as there is no communication delay. The following question then arises: given a communication graph $G=(\cV,\cE)$ with communication delays $\tau_\edgevw$ and computation delays $\tau_v$ for $v\in\cV$ and $\edgevw\in\cE$, can we reverse-engineer and build communication/computation schemes that fit in the AGRAF SGD framework and that do not break the communication and computation constraints? Can we analyze such a scheme and prove that it mixes well (in the sense that \Cref{hyp:consensus} holds, for explicit values of $\rho,k_\rho$) ?

\myparagraph{Overview of the Loss Network scheme.}
Starting with $\cI_k=\set{v_k,w_k}\in\cE$, and communication matrices $W_k$ corresponding to an averaging along the edge $\set{v_k,w_k}$ as a baseline (\textit{i.e.}, $W_k=I_\cV-\frac{(e_{v_k}-e_{w_k})(e_{v_k}-e_{w_k})^\top}{2}$ where $(e_v)$ is the canonical basis of $\R^\cV$) as a baseline,  choosing a sequence $\cI_k$ such that there is no induced communication delays becomes tricky. While assuming that $\set{v_k,w_k}$ is sampled independently from the past with fixed probability \citep{lian2017asynchronous} is amenable for the analysis (since then \Cref{hyp:consensus} directly holds for $k_\rho=1$), this can incur communication delays if for instance the same node is sampled in two consecutive updates.

To alleviate this issue, we remove the independence between sampled edges in the following way: we impose that nodes that are already involved in a communication are tagged as \textit{busy}, and that busy nodes cannot be involved in new communications. Then once a node finishes a computation, it can then choose a new neighbor (\textit{who is not busy}) to start communicating with. Doing so, the induced communication matrices are no longer independent, as they follow a Markov process.
This scheme is inspired by \textbf{Loss-Networks}, introduced in \citep{lossnetworks1991kelly} to model telecommunication networks, in which an edge in the graph models a phone communication that can happen; since a phone cannot make several calls in parallel, once involved in a communication with some neighboring node it cannot be called by another neighbor while it is busy; this is exactly the same process we use, phone calls being replaced by model communications. 

\myparagraph{How to schedule such a process ?}
If nodes start a new communication right after they finish their last one, the process can end up in deadlock and thus does not mix at all: this is for instance the case on the cycle or line graphs with an even number of nodes \citep{lossnetworks1991kelly}.
We thus need to introduce some randomness and some waiting times.
We proceed as follows and use exponential random waiting times as in \citet{lossnetworks1991kelly}.

\textbf{(i)} Once a node $v$ finishes a communication, it waits a time $T_v\sim\mathrm{Exp}(p_v)$ (exponential random variable, of intensity $p_v$). 

\textbf{(ii)} If $v$ is still not busy after this waiting time, $v$ samples some neighboring node $w\sim v$ with probability $\frac{p_\edgevw}{p_v}$ to communicate with, for $\sum_{w\sim v}p_\edgevw=p_v$.

\textbf{(iii)} If $w$ is busy, this procedure restarts at \textbf{(i)}, else both $v$ and $w$ become busy and can communicate. Once they are busy, they cannot communicate with other nodes. The communication between $v$ and $w$ consists in averaging local values by setting $x_v,x_w$ to $(x_v+x_w)/2$. When this is done, they each perform a local (eventually delayed) gradient step, and then become \textit{non-busy}. Overall, the $k^{th}$ update reads:
\begin{equation}\label{eq:LN_update}
    x_{v_k}^{k+1} = \frac{x_{v_k}^k + x_{w_k}^k}{2} - \gamma \nabla F_{v_k}\big(x_{v_k}^{k-{\tau(v_k,k)}},\xi_{v_k}^{k-{\tau(v_k,k)}}\big)\,,
\end{equation}
and similarly at node $w_k$.
The procedure described just above (\textbf{(i)-(ii)-(iii)}) to sample pairs of nodes that iteratively perform computations and pairwise communications can be instantiated locally, provided nodes know when their neighbors in the graph are busy --- this can be relaxed by adding some ‘‘busy-checking'' operation.
However, the key challenge here lies in that the communication matrices $(W_k)_{k\geq0}$ induced by the updates \Cref{eq:LN_update} are not independent, and analyzing some form of ergodic mixing time becomes highly non-trivial. Still, using the randomness introduced in this procedure through the exponential waiting times and the sampling of neighbors, we are able to prove that \Cref{hyp:consensus} holds, for values of $\rho,k_\rho$ that depend on the physical delays.

\begin{assumption}[Loss Network assumptions]\label{hyp:LN}
    There exist $\tau_v,\tau_\edgevw\in\R_{>0}$\footnote{$\tau_v,\tau_\edgevw$ are \textbf{physical continuous-time} delays.} for $v\in\cV,\edgevw\in\cE>0$ such that a communication between $v$ and $w$ takes a time at most $\tau_\edgevw$, and computing a stochastic gradient at node $v$ takes a time at most $\tau_v$.
\end{assumption}

\begin{theorem}\label{thm:LN}
    Under \Cref{hyp:LN}, assume that $p_\edgevw=\min\left(\frac{1}{\max_{u\sim v}\tau_\edgeuw'}, \frac{1}{2(\max(d_v,d_w)-1)\tau_\edgevw'} \right)$, where $d_v$ is the degree of node $v$ and $\tau'_\edgevw=\tau_\edgevw+\max(\tau_v,\tau_w)$.
    Let $\Lambda$ be the spectral gap (smallest non-null eigenvalue of the weighted Laplacian) of the graph $G$ with weights $\lambda_\edgevw=\frac{\min_{u\sim\edgevw}p_\edgevw}{d\sum_{e\in\cE}p_e},\,\edgevw\in\cE$, where $d$ is the max degree in the graph.
    Then, \Cref{hyp:consensus} is verified for $\frac{\rho}{k_\rho}=\Tilde\cO(\Lambda)$. 
\end{theorem}

Given a graph $G$ with physical communication and computation latencies $\set{\tau_v,\tau_\edgevw}$ (\Cref{hyp:LN}), we are thus able to exhibit a communication scheme that satisfies communication and computation constraints, while still fitting in the framework of AGRAF SGD under the assumptions used in our convergence rates.
Crucially, the mixing constant $\Lambda$ explicitly depends on the graph and the delays, through the smallest non-null eigenvalue of the weighted graph Laplacian, with explicit weights $\lambda_\edgevw$ on the edges.
These weights depend on \textbf{local} delays: having straggler nodes or edges do not slow down communication or computations, if there are fast edges/nodes that are dense enough in the graph. 
To further highlight the importance of having weights $\lambda_{\edgevw}$ that only depend on the local delays, this can be put in perspective of Asynchronous SGD, that is proved to depend only on the averaged computation delay $\frac{1}{n}\sum_{v\in\cV}\frac{1}{\tau_v}$ rather than the max delay \citep{koloskova2022sharper,mishchenko2018delay}. For decentralized optimization over a given graph, depending on the averaged communication delays wouldn't make sense since all communication paths need to be taken into account; hence, the counterpart to the mean delay in the graph is a \textbf{weighted} Laplacian, with weights on edge $\edgevw$ that are function of \textbf{local} delays, instead of a max delay which is the asynchronous speedup \citep{even2021delays}.
\myparagraph{Disclaimer.}
The proof of \Cref{thm:LN} is adapted from that of \cite{even2021asynchrony}, an unpublished work by a subset of the authors.

\section*{Conclusion}
We introduced a unifying framework for studying asynchronous and decentralized algorithms; our analysis recovers and improves over that of previous asynchronous decentralized SGD works, while being much more general. The flexibility of our framework furthermore enables us to leverage an asynchronous speedup under communication and computation delays, by the introduction of Loss Networks and new analysis tools, thus providing a non-trivial sampling scheme that still satisfies the ergodic mixing property introduced by~\citet{koloskova2020unified}.

\myparagraph{Aknowledgements.}
M.E. thanks Konstantin Mischenko for initiating discussions and suggesting this subject (asynchronous SGD on graphs) and for all the valuable discussions. A.K. and M.E. also thank Martin Jaggi for interesting discussions.

\bibliographystyle{plainnat}

\bibliography{refs.bib,refsASGD.bib}

\newpage
\onecolumn

\appendix

\section{Equivalence of two ergodic mixing assumptions}

The following assumption is a consequence of \Cref{hyp:consensus}: if \Cref{hyp:consensus} holds for some $\tau,\rho$, then \Cref{hyp:consensus2} holds for $\bar\rho=c\frac{\rho}{\tau}$ where $c$ is some numerical constant. In fact, as we prove in \Cref{prop:equivalence_mixing}, they are both equivalent, but the following proves to be easier to handle in the analysis.

\begin{assumption}\label{hyp:consensus2}
    $W_k\one=\one$ and  there exist $\bar\rho$ such that we have $\forall k,\ell\in\N$ and $\forall  \xx\in\R^\cV$:
\begin{equation}\label{eq:hyp_consensus2}
\begin{aligned}
    &\esp{\NRM{W^{(k:k+\ell)} \xx-\frac{1}{n}\one\one^\top \xx}^2|\cF_k}\\
    &\quad\leq 2(1-\bar\rho)^{2\ell}\NRM{ \xx-\frac{1}{n}\one\one^\top \xx}^2\,.
\end{aligned}
\end{equation}
\end{assumption}
\begin{proposition}\label{prop:equivalence_mixing}
Assumptions~\ref{hyp:consensus} and~\ref{hyp:consensus2} are equivalent, in the following sense.
\begin{enumerate}
    \item If Assumption~\ref{hyp:consensus} holds for some $\rho\in[0,1]$ and for some $k_\rho\in\N^*$, then Assumptions~\ref{hyp:consensus2} holds for $\bar \rho = c\frac{\rho}{k_\rho}$, for $c>0$ some numerical constant.
    \item If Assumption~\ref{hyp:consensus2} holds for some $\bar\rho\in[0,1]$, then Assumption~\ref{hyp:consensus} holds for any $\rho\in(0,1)$ and $k_\rho=\left\lceil\frac{\frac{1}{2}\ln(2) \ln(1-\rho)}{\ln(1-\bar\rho)} \right\rceil$ ($\propto \frac{\rho}{\bar\rho}$ for $\rho,\bar\rho$ small).
\end{enumerate}
\end{proposition}

\begin{proof}
    We first prove \textbf{1.} Assume that Assumption~\ref{hyp:consensus} holds for some $\rho,k_\rho$.
    If \Cref{hyp:consensus} holds for $\rho$ it holds for any $\rho'<\rho$, so that we can assume without loss of generality that $\rho\leq 1-\sqrt{2}$.
    Let $k,\ell\in\N$ and $\xx\in\R^{\cV}$.
    Using Assumption~\ref{hyp:consensus} $\lfloor \frac{\ell}{k_\rho}\rfloor$, we have that:
    \begin{equation*}
        \esp{\NRM{W^{(k:k+\ell)} \xx-\frac{1}{n}\one\one^\top \xx}^2|W_0,\ldots,W_k}\leq (1-\rho)^{2\lfloor \ell/k_\rho\rfloor}\NRM{ \xx-\frac{1}{n}\one\one^\top \xx}^2\,.
    \end{equation*}
    Thus, $(1-\rho)^{2\lfloor \ell/k_\rho\rfloor}\leq (1-\rho)^{2 (\ell/k_\rho-1)}\leq \frac{1}{(1-\rho)^2}(1-\rho)^{2\ell/k_\rho}$. 
    Then, $\frac{1}{(1-\rho)^2}\leq 2$ and $(1-\rho)^{2\ell/k_\rho}\leq e^{-2\ell \rho/k_\rho}\leq (1- c\frac{\rho}{k_\rho})^{2\ell}$ for $c\in(0,1)$ some numerical constant ($c=\frac{e-1}{e}$), since $\frac{\rho}{k_\rho}\leq 1$.

    We now prove \textbf{2.} Assume that Assumption~\ref{hyp:consensus2} holds for $\bar\rho>0$, and let $\rho>0$.
    We have:
    \begin{equation*}
        \esp{\NRM{W^{(k:k+\ell)} \xx-\frac{1}{n}\one\one^\top \xx}^2|W_0,\ldots,W_k}\leq (1-\rho)^{2}\NRM{ \xx-\frac{1}{n}\one\one^\top \xx}^2\,,
    \end{equation*}
    provided that $\ell$ satisfies:
    \begin{equation*}
        2(1-\bar\rho)^{2\ell}\leq (1-\rho)^2\,.
    \end{equation*}
    This is satisfied for:
    \begin{equation*}
        \ell\geq \frac{\frac{1}{2}\ln(2) \ln(1-\rho)}{\ln(1-\bar\rho)}\,,
    \end{equation*}
    and thus Assumption~\ref{eq:hyp_consensus} holds for $\rho$ and $k_\rho=\left\lceil\frac{\frac{1}{2}\ln(2) \ln(1-\rho)}{\ln(1-\bar\rho)} \right\rceil$.
\end{proof}
\newpage

\section{Preliminaries for our convergence rates}

For $k\geq0$, , and for any $k\geq0$ and $v\in\cV$:
\begin{equation*}
    \nex(k,v)=\inf\set{\ell\geq k\,,\, v\in\cI_\ell}\,,\quad \prev(k,v)=\sup\set{\ell<k\,,\, v\in\cI_\ell}\cup\set{0}\,,\quad \tau(k,v)=k-\prev(k+1,v)\,.
\end{equation*}
In other words, at a given iteration $k$, $\nex(k,v)$ is the iteration at which the node $v$ will finish computing its current gradient, $\prev(k, v)$ is the iteration at which the node $v$ started computing its current gradient, and $\tau(k,v)$ is the current computational delay of node $v$ at time $k$.

Let also $\bar x^k=\frac{1}{n}\sum_{v\in\cV} x_v^k\in\R^d$ and $\gg^k=(\one_{v\in\cI_k} \nabla F_v(x_v^{\prev(k,v)},\xi_v^{\prev(k,v)})$, so that $\xx^{k+1}=W_k\xx^k-\gamma\gg^k$.

\subsection{Virtual iterate sequence to handle delays} \label{sec:virtual_sequence}

As in \cite{asynchronous_sgd_mischenko}, the delay analysis relies on the study of a virtual sequence.
Noticing that $\bar x^{k+1}=\bar x^k-\frac{\gamma}{n} \sum_{v\in\cI_k} g_v^{t-\tau(k,v)}$ and mimicking the analysis of asynchronous SGD, we introduce the sequence $\set{\hat x^k,k\geq1}$ that lives in $\R^d$, defined through the following recursion:
\begin{equation*}
    \hat x^{k+1}=\hat x^k-\frac{\gamma}{n} \sum_{v\in\cI_k} g_{v}^k\,,\quad \hat x_1=\bar x_0-\frac{\gamma}{n}\sum_{v\in\cV}  g_v^0\,.
\end{equation*}
We then have, for all $k\geq1$:
\begin{equation*}
    \hat x^k-\bar x^k=-\frac{\gamma}{n}\sum_{v\in\cV\setminus\cI_k}  g^{\prev(k,v)}\,.
\end{equation*}
The difference $\NRM{\hat x^k-\bar x^k}$ can thus be easily bounded. 

\begin{lemma}[Virtual iterates control]\label{lem:virtual_iterates}
    If stochastic gradients are bounded by a constant $B>0$, we have:
\begin{equation}\label{eq:bound_virtual_B}
    \NRM{\hat x^k-\bar x^k}\leq \gamma B\,.
\end{equation}
In the general case,
\begin{equation}\label{eq:bound_virtual_gen}
    \esp{\NRM{\hat x^k-\bar x^k}^2}\leq \frac{2\gamma^2}{n}\left( \sigma^2 + \sum_{v\in\cV} \esp{\NRM{\nabla f_v(x_v^{\prev(v,k)})}^2}\right).
\end{equation}
\end{lemma}

\begin{proof}
    \Cref{eq:bound_virtual_B} is proved using a triangle inequality, while \Cref{eq:bound_virtual_gen} is a direct application of \cite[Lemma 15]{stich2019error}.
\end{proof}

\subsection{Consensus control}

\begin{lemma}[Consensus control]\label{lem:consensus_control}
    We have:
    \begin{align}\label{eq:bound_consensus_smooth}\sum_{k<K}\E\NRM{ \xx^k-\bar \xx^k}^2&\leq 2\gamma^2\sigma^2\brhom \NK + \frac{4\gamma^2}{\bar\rho^2}\sum_{k<K}\sum_{v\in\cI_k}\esp{\NRM{\nabla f_v(x_v^{k-\tau(k,v)})}^2} \\
        &\leq 2\gamma^2\sigma^2\brhom \NK + \frac{4\gamma^2}{\bar\rho^2}\sum_{k<K}\sum_{v\in\cI_k}\esp{\NRM{\nabla f_v(x_v^k)}^2}\,.
    \end{align}
    If the stochastic gradients are bounded by some $B>0$,
    \begin{equation}\label{eq:bound_consensus_B}
        \sum_{k<K}\E\NRM{ \xx^k-\bar \xx^k}^2 \leq \frac{2\gamma^2B^2}{\bar\rho^2}\NK\,.
    \end{equation}    
\end{lemma}

\begin{proof}
Under Assumption~\ref{hyp:consensus}, we can bound the variations of $ \xx^k-\bar\xx^k$ (here, $\bar \xx^k= \one \one^\top\xx^k$).
Using Cauchy-Schwarz inequality, for $a_m>0$ scalars and $b_m\in\R^p$ vectors, we have:
\begin{equation*}
    \NRM{\sum_m b_m }^2 \leq \left(\sum_m a_m^{-1} \right)\left(\sum_m a_m \NRM{b_m}^2\right)\,.
\end{equation*}
We now apply this to $\xx^k-\bar \xx^k = -\gamma\sum_{m=0}^kW^{(m:k)}(  \tilde\gg^{m}-\bar \tilde \gg^{m})$ to obtain:
\begin{align*}
    \E\NRM{ \xx^k-\bar \xx^k}^2&=\esp{\NRM{\gamma\sum_{m=0}^kW^{(m:k)}( \tilde \gg^{m}-\bar \tilde \gg^{m})}^2}\\
    &\leq \gamma^2\sum_{m'=0}^k(1-\bar\rho)^{k-m'}\sum_{m=0}^k(1-\bar\rho)^{-(k-m)}\esp{\NRM{W^{(m:k)}( \tilde \gg^{m}-\bar \tilde \gg^{m})}^2}\\
    &\leq 2\gamma^2\frac{1}{\bar\rho}\sum_{m=0}^k(1-\bar\rho)^{k-m}\esp{\NRM{ \tilde \gg^{m}}^2}\\
\end{align*}
leading to, if stochastic gradients are bounded by $B$:
\begin{equation*}
    \E\NRM{ \xx^k-\bar \xx^k}^2 \leq \frac{2\gamma^2B^2}{\bar\rho}\sum_{\ell<k}(1-\bar\rho)^{k-\ell}\nl\,,
\end{equation*}
and thus:
\begin{equation*}
    \sum_{k<K}\E\NRM{ \xx^k-\bar \xx^k}^2 \leq \frac{2\gamma^2B^2}{\bar\rho^2}\NK\,.
\end{equation*}
We also have, using a bias-variance decomposition (not exactly, since the $\gg^m$ are not independent, but using the martingale version as in \cite[Lemma 15]{stich2021errorfeedback}):
\begin{align*}
    \E\NRM{ \xx^k-\bar \xx^k}^2&=\esp{\NRM{\gamma\sum_{m=0}^kW^{(m:k)}(  \tilde\gg^{m-\tau(m)}-\bar\tilde  \gg^m)}^2}\\
    &\leq 2\gamma^2\sigma^2\sum_{\ell<k}(1-\bar\rho)^{k-\ell}\nl + \frac{4\gamma^2}{\bar\rho}\sum_{m=0}^k(1-\bar\rho)^{k-m}\sum_{v\in\cI_m}\esp{\NRM{\nabla f_v( x^{(m-\tau(m,v))}_{v})}^2}\,,
\end{align*}
so that:
\begin{equation*}
    \sum_{k<K}\E\NRM{ \xx^k-\bar \xx^k}^2\leq 2\gamma^2\sigma^2\brhom \NK + \frac{4\gamma^2}{\bar\rho^2}\sum_{k<K}\sum_{v\in\cI_k}\esp{\NRM{\nabla f_v(x_v^{k-\tau(k,v)})}^2}\,.
\end{equation*}
\end{proof}

\section{Loss Networks analysis}

\myparagraph{Disclaimer.}
This proof is adapted from that of \cite{even2021asynchrony}, an unpublished work by a subset of the authors.

In this section, we prove \Cref{thm:LN} and provide some more information on loss networks.
The updates of \textit{decentralized SGD on loss networks} write as:
\begin{equation}\label{eq:LN:updates-app}
\left\{
 \begin{aligned}
    x_{v_k}^{k+1} &= \frac{x_{v_k}^k + x_{w_k}^k}{2} - \gamma \nabla F_{v_k}\big(x_{v_k}^{k-{\tau(v_k,k)}},\xi_{v_k}^{k-{\tau(v_k,k)}}\big)\\
    x_{w_k}^{k+1} &= \frac{x_{v_k}^k + x_{w_k}^k}{2} - \gamma \nabla F_{w_k}\big(x_{w_k}^{k-{\tau(w_k,k)}},\xi_{w_k}^{k-{\tau(w_k,k)}}\big)
\end{aligned}   
\right.\,,
\end{equation}
leading to $\xx^{k+1}=W_k\xx^k-\gamma\gg^k$, for $W_k=W_{\set{v_k,w_k}}=I_\cV-\frac{(e_{v_k}-e_{w-k})(e_{v_k}-e_{w-k})^\top}{2}$, and $\gg^k$ the corresponding delayed gradients.
Note then that this takes the same form as the \textbf{AGRAF SGD} sequence.

\begin{definition}[Poisson point process (P.p.p.)]
    A \textbf{Poisson point process} of intensity $p>0$ is a random discrete subset $\cP$ of $\R_{\geq0}$ that can be written as $\cP=\set{T_0<T_1<\ldots<T_k<\ldots}$, where $(T_{k}-T_{k-1})_{k\geq 1}$ are \textit{i.i.d.} exponential random variables of mean $\frac{1}{p}$. 
\end{definition}

\cite{boyd2006gossip} consider a model (without any delay) for gossip algorithms, where updates are that of \Cref{eq:LN:updates-app} without the gradient steps, and these updates happen at the times of Poisson point processes (a \textit{P.p.p.} of intensity $p_\edgevw$ for an update along $\edgevw$). Consequently, $W_k$ is independent from the past, and $\proba{W_k=W_\edgevw}\propto p_\edgevw$.

The \emph{P.p.p.~model} considered in \cite{boyd2006gossip} where the updates are performed at the times of \textit{Poisson point processes} is particularly amenable to analysis, but it assumes that communications and computations are done instantaneously. Thus, actual implementations differ from its underlying assumptions, unless further synchrony is assumed. To alleviate this issue, with pairwise communications ruled by point processes as a baseline, we consider a protocol in which nodes are tagged as \emph{busy} when they are already engaged in an update, and communications between busy nodes are forbidden. Our model is inspired from classical Loss Network models \citep{lossnetworks1991kelly}, in which edges are activated following the same procedure as in the \emph{P.p.p.~model}, with a \emph{P.p.p.}~of intensity $p_{\edgevw}$. Note that we do not consider these intensities to be constraints of the problem, but rather parameters of the algorithm, that can be tuned. Each node has an exponential clock of intensity $p_v\frac{1}{2}\sum_{w\sim v}p_{\edgevw}$. At each clock-ticking, if $v$ is not busy, it selects a neighbor $w$ with probability $p_{\edgevw}/\sum_{u\sim v}p_{\edgeuv}$. 
If $w$ is not \emph{busy}, $v$ and $w$ compute and exchange information, becoming busy for a duration $\tau_{\edgevw}'$. We can think of this procedure as classical gossip on an underlying random graph that follows a Markov-Chain process. 
The difference between our communication model on Loss Networks and the P.p.p. model lies in that in our case, $W_k$ is not independent on the past. In fact, we have:
\begin{equation*}
    \proba{\set{v_k,w_k}=\edgevw | \cF_k} = \frac{\one_\set{v,w\text{ not busy at time }T_k}p_\edgevw}{\sum_{\set{u,u'}\in\cE}\one_\set{u,u'\text{ not busy at time }T_k}p_\set{u,u'} }\,,
\end{equation*}
leading to complicated intricacies between the matrices $(W_k)_k$, that we need to handle.

Proving \Cref{thm:LN} requires to show that there exist $\rho,k_\rho$ (that need to be computed) such that for any $k\geq 0$, $\xx\in\R^\cV$,
\begin{equation*}
    \esp{\NRM{W_{\set{v_{k+k_\rho-1},w_{k+k_\rho-1}}} \cdot\ldots\cdot W_\edgevkwk (\xx-\bar \xx)}^2|\cF_k }\leq (1-\rho)^2 \NRM{\xx-\bar\xx}^2\,.
\end{equation*}

Our proof of \Cref{thm:LN} follows three main steps: \emph{i)} Deriving convergence results for more general communication schemes than loss networks, under deterministic assumptions on the activations. \emph{ii)} Adapting Step i) to stochastic assumptions on the delays. \emph{iii)} Deriving high-probability upper-bounds on the delays between two activations in loss networks in order to fall under the assumptions of Step i).

\subsection{Descent lemma under deterministic assumptions on the activations}

We consider general activation processes $\cP_{\edgevw}$, where we define $\cP_\edgevw$ as $\cP_\edgevw=\set{T_k\,:\, \edgevkwk=\edgevw}$, and these times are called \textbf{activation times of edge \edgevw}. 
When edge $\edgevw$ is activated, the update described in \eqref{eq:LN:updates-app} is performed.
The delay of an edge is defined as its (random) waiting time between two activations.
Two ergodicity-like conditions on the delays are needed: \textit{(i) edges activated regularly enough and (ii) incident edges must not be activated too many times.}

We now formally introduce these assumptions. We consider discrete time in this section: more precisely, $k\in \N$ stands for the $k$-th edge activation.

\begin{definition}\label{def:quant_delays} Consider a communication scheme with edge-activation point processes $\cP_{\edgevw}$. Let $k=0,1,2,...$ index the consecutive edge activations. Let $\ell\in \N$, $\edgevw$ and $\edgeuu \in E$. Let $k_{\edgevw}<\ell_{\edgevw}$ such that $k_{\edgevw}\leq k <\ell_{\edgevw}$ be consecutive activation times (in discrete time) of $\edgevw$. Denote $T_{\edgevw}(k)=\ell_{\edgevw}-k_{\edgevw}-1$ the total number of edge activations between the two consecutive activations of $\edgevw$. Denote $N(\edgeuu,\edgevw,k)$ the number of activations of edge $\edgevw$ in the activations $\{s_{\edgevw},s_{\edgevw}+1,...,t_{\edgevw}-1\}$.
\end{definition}

\begin{assumption}[Delay Assumptions] \label{hyp} There exist $T\in \N^*$, $a,b>0$, and $\ell_{\edgevw}>0,\edgevw\in E$ such that, for the quantities and the communication scheme in Definition \ref{def:quant_delays}:
\begin{enumerate}
    \item For all $k \in \N$, all edges are activated between iterations $k$ and $k+T-1$.
    \item $\forall k\geq0, \forall (\edgevw)\in E, T_{\edgevw}(k)\leq a \ell_{\edgevw}$: $(\edgevw)$ is activated at least every $a\ell_{\edgevw}$ activations.
    \item $\forall k\geq 0, \forall (\edgevw),(\edgeuu)\in E$ such that $(\edgeuu)\sim(\edgevw)$, $N(\edgeuu,\edgevw,k)\leq \lceil \frac{b\ell_{\edgevw}}{\ell_{\edgeuu}}\rceil$.
\end{enumerate}
\end{assumption}
Assumption (1) is implied by Assumption (2) if $T=\max_{(\edgevw)}\ell_{\edgevw}$. Taking $\ell_{\edgevw}$ as a deterministic upper-bound on the delays of edge $(\edgevw)$ between two activations in continuous time is sufficient to have Assumption (2) and (3), with some normalizing constant $a$, and $b$ such that $\ell_{\edgevw}/b$ is a lower-bound on these delays.

The main technical difficulty lies in the fact that at a defined activation time $t$, some nodes are not available: at any time $k\ge 0$, $\sum_{(\edgevw)\in E \text{ not busy}} W_k$ usually differs from $\sum_\edgevw p_\edgevw W_\edgevw$ (and $\sum_{(\edgevw)\in E \text{ not busy}} W_k$ may have a null spectral gap) as in \emph{Markov-Chain Gradient Descent} \citep{even2023mcsgd}, thus making an analysis such as in the \emph{P.p.p. model} impossible.
To alleviate this difficulty, in order to make sure that all edges are taken into account when performing the averaging, the Lyapunov function $\Lambda_k$ that we study considers the value of the objective for $T$ consecutive activation times. It is defined as follows:
\[\forall k\in \N, \Lambda_k(\xx)=\frac{1}{T}\sum_{\ell=k}^{k+T-1} \NRM{W^{(0,\ell)}(\xx-\bar\xx)}^2\,,\quad \xx\in\R^\cV  \,.\] 
The first step of the proof of Theorem \ref{thm:LN} consists in proving the following.
\begin{theorem} 
\label{thm:LN_det}
Consider a general communication scheme as in Definition \ref{def:quant_delays}, that satisfies Assumption \ref{hyp} for constants $\ell_{\edgevw},a,b>0,$. Let $\gamma$ be the smallest positive eigenvalue of the Laplacian of the graph $G$ with weights:
\begin{equation*}\nu_{\edgevw}=C\ell_{\edgevw}^{-1}\min_{\edgeuu\sim \edgevw}\frac{\ell_{\edgeuu}}{\ell_{\edgevw}}\,,\quad \edgevw\in\cE\,,
\end{equation*}
where $C=\frac{1}{2a+8d_{\max}^2ab}$. Then we have, for all $k,\ell\in \N$:
\begin{equation*}
    \Lambda_{k+\ell}(\xx)\leq \left(1-\gamma\right)^{\ell} \Lambda_k(\xx)\,.
\end{equation*}
\end{theorem}

\begin{proof}We fix $\xx\in\R^\cV,k,\ell$.
To prove this intermediate theorem, we need to study every matrix multiplication involved. At iteration $k$, not every coordinates is available, hence the need to study the impact of $T$ multiplications together.

A gradient step alongside edge $\edgevw$ only involves edges in its neighborhood (thanks to the sparsity of the matrix $A$), a key element that will need to be explicited. The proof involves three main steps.\\

Before that, we need to introduce \textbf{edge dual variables}.
Matrix multiplications by matrices like $W_\edgevw$ aim at minimizing the function $F(\yy)=\frac{1}{2}\sum_{v\in\cV}(y_v-x_v)^2$, which is minimized at $\yy=\bar\xx$.
A standard way to deal with the constraint $x_1=...=x_n$, is to use a dual formulation, by introducing a dual variable $\lambda\in\R^\cE$ indexed by the edges. We first introduce a matrix $A\in \R^{\cV \times \cE}$ such that $\rm Ker (A^\top)=Vect(\mathbb{I})$ where $\mathbb{I}$ is the constant vector $(1,...,1)^\top$. $A$ is chosen such that:
\begin{equation}\label{eq:matrixA}
    \forall \edgevw\in E, A e_{\edgevw} = \mu_{\edgevw} (e_v-e_w).
\end{equation} for some non-null constants $\mu_{\edgevw}$. We define $\mu_{\edgevw}=-\mu_{\edgevw}$ for this writing to be consistent. This matrix $A$ is a square root of the laplacian of the graph weighted by $\nu_{\edgevw}=\mu_{\edgevw}^2$.
The constraint $x_1=...=x_n$ can then be written $A^\top x=0$. The dual problem reads as follows:
\begin{align*}
    &\min_{\yy \in \R^\cV,A^\top \yy=0} F(\yy) =\min_{\yy \in \R^{\cV}} \max_{\lambda \in \R^\cE} F(\yy) -\langle A^\top \yy,\lambda\rangle.
\end{align*} 
Let $F_A^*(\lambda):=F^*(A\lambda)=F_A(\lambda)$ for $\lambda \in \R^{E\times d}$ where $F^*$ is the Fenchel conjugate of $F$.
Now, notice that for our particular form of $F$, we in fact have $F^*=F$.
The dual problem reads
\begin{equation*}
    \min_{\yy \in \R^{\cV},y_1=...=y_n} F(\yy) =\max_{\lambda\in \R^{\cE}} -F_A(\lambda).
\end{equation*}
Thus $F_A^*(\lambda)$ is to be minimized over the dual variable $\lambda \in \R^{\cE}$. 

We now make a parallel between pairwise operations between adjacent nodes in the network and coordinate gradient steps on $F_A^*$. As $F_A^*(\lambda)=\max_{\yy\in \R^{\cV}} -F(\yy)+\langle A\lambda,\yy\rangle$, to any $\lambda\in \R^{\cE}$ a primal variable $\yy\in \R^{\cV}$ is uniquely associated through the formula $\nabla F(\yy)=A\lambda$.
The partial derivative of $F^*_A$ with respect to coordinate $\edgevw\in\cE$ of $\lambda$ reads :
\begin{align*}
    \nabla_{\edgevw} F_A^*(\lambda)&=(A e_{\edgevw})^\top \nabla F^*(A\lambda)=\mu_{\edgevw}(\nabla g_v^*((A\lambda)_v)- \nabla g_w^*((A\lambda)_w))\,,
\end{align*}
where we denote $g_v(y):\frac{1}{2}(y-x_v)^2$.
Consider then the following step of coordinate gradient descent for $F^*_A$ on coordinate $\edgevw$ of $\lambda$, performed when edge $\edgevw$ is activated at iteration $k$ (corresponding to time $T_k$), and where $U_{\edgevw}=e_{\edgevw}e_{\edgevw}^\top$:
\begin{equation}\label{eq:step_lambda}
    \lambda_{k+1}=\lambda_{k+1}-\frac{1}{\mu_{\edgevw}^2}U_{\edgevw}\nabla_{\edgevw} F_A^*(\lambda_{k}).
\end{equation}
Denoting $\yy_k=A\lambda_{k}\in \R^{\cV}$, we obtain the following formula for updating coordinates $v$ and $w$ of $\yy$ when $\edgevw$ activated:
\begin{align}
&y_{v,k+1}=y_{v,k}-\frac{1}{2}(y_{v_k}-y_{w_k})=\frac{1}{2}(y_{v_k}y_{w_k}) = y_{w,k+1}\,.
\end{align}
Thus, $\yy^{k+1}=W_k\yy^k$ is equivalent to $\lambda_{k+1}=\lambda_k-\frac{1}{2\mu_\edgevkwk^2}\nabla_\edgevkwk F_A^*(\lambda_k)$, which is easier to study.
Also, notice that this is the consensus distance exctly: $F_A^*(\lambda)=F(\yy)$ for $\yy=A\lambda$.

Hence, $\Lambda_k(\xx)=F(\yy^k)=F_A^*(\lambda_k)$ here $\yy^k=A\lambda^k$ is obtained with the recursion $\lambda^{k+1}=\lambda^k-\frac{1}{2\mu_\edgevkwk^2}\nabla_\edgevkwk F_A^*(\lambda^k)$, with initialisation $\yy^0=\xx$: we thus study this sequence.

\noindent \textbf{Step 1:}
First, notice that $F_A^*$ is $\mu_\edgevw^2$-smooth along every coordinate $\edgevw$, so that using local smoothness, for all $\edgevw\in\cE$ and $\lambda\in\R^\cE$, for $\gamma\leq \frac{1}{2\mu_\edgevkwk^2}$, we have:
\begin{equation}\label{eq:local_smoothness}
    F_A^*(\lambda-\nabla_\edgevw F_A^*(\lambda)) -F_A^*(\lambda)\leq \frac{1}{4\mu^2_\edgevw}\NRM{\nabla_\edgevw F_A^*(\lambda)}^2\,.
\end{equation}
Applying \Cref{eq:local_smoothness}, where $\edgevlwl$ is the $\ell^{th}$ activated edge: 
\begin{equation}
    F_A^*(\lambda^{\ell+1})-F_A^*(\lambda^\ell) \leq -\frac{1}{4\mu_{\edgevlwl}^2}\|\nabla_{\edgevlwl}F_A^*(\lambda^\ell)\|^2\,.
\end{equation} 
Hence, summing: 
\begin{equation}
    \Lambda_{k+1}\leq \Lambda_k - \frac{1}{T}\sum_{k\leq \ell < k+T} \frac{1}{4\mu_{\edgevlwl}^2}\|\nabla_{\edgevlwl}F_A^*(\lambda^\ell)\|^2\,,
\end{equation}
Notice that:
\begin{equation}
    \frac{1}{T}\sum_{k\leq \ell < k+T} \sum_{\edgevw\in \cE} \|\nabla_{\edgevw}F_A^*(\lambda^\ell)\|^2 = \frac{1}{T}\sum_{k\leq \ell < k+T}\|\nabla F_A^*(\lambda^\ell)\|^2 \geq \sigma_A \Lambda_t
\end{equation} 
$ \sigma_A$ is the strong convexity parameter of $F_A^*$ which is equal tolower bounded by $\lambda_{min}^+(A^TA)$, which itself is exactly the smallest positive non-null eigenvalue of the graph Laplacian with weights $\mu_\edgevw^2$. Hence, if an inequality of the type
\begin{equation}\label{neededthm2}
    \frac{C}{T}\frac{1}{T}\sum_{k\leq \ell < k+T} \sum_{\edgevw\in \cE} \|\nabla_{\edgevw}F_A^*(\lambda^\ell)\|^2\leq \frac{1}{4\mu_{\edgevlwl}^2}\|\nabla_{\edgevlwl}F_A^*(\lambda^\ell)\|^2
\end{equation}
holds, we have using strong convexity:
\begin{equation}
    \Lambda_{k+1}\leq \Lambda_k-\frac{C}{T}\sum_{k\leq \ell < k+T} \|\nabla F_A^*(\lambda^\ell)\|^2\leq (1-C\sigma_A)\Lambda_k\,.
\end{equation}
We thus need to tune correctly the $\mu_{\edgevw}^2$ and $C$ in order to have \eqref{neededthm2} verified.\\

\noindent \textbf{Step 2:} We are looking for necessary conditions for \eqref{neededthm2} to hold. In the left term, every coordinate is present at each time $\ell$. However, in the right hand side of the inequality, just the activated one is present. We will need to compensate this with a bigger factor in front of the gradients. In order to compare these quantities, we need to introduce upper bound inequalities on $\|\nabla_{\edgevw}F_A^*(\lambda(s))\|^2$, that only make activated coordinates intervene. Let $s\in \{t,...,t+T-1\}$, and suppose that there exists $t\leq r \leq s < r+t_{\edgevw} \leq t+T-1$ such that $\edgevw$ is activated at times $r$ and $r+t_{\edgevw}$. Thanks to the asumption on $T$, either one of these integers exists. If the other one doesn't, replace it with $t$ for $r$, and by $t+T-1$ for $r+t_{\edgevw}$. Thanks to our asumptions, we know that $t_{\edgevw}\leq a\ell_{\edgevw}$. We have the following basic inequalities:
\begin{align}
    \|\nabla_{\edgevw}F_A^*(\lambda(s))\|^2 &\leq (\|\nabla_{\edgevw}F_A^*(\lambda(r))\|+\|\nabla_{\edgevw}F_A^*(\lambda(s))-\nabla_{\edgevw}F_A^*(\lambda(r))\|)^2\\
    &\leq 2(\|\nabla_{\edgevw}F_A^*(\lambda(r))\|^2+\|\nabla_{\edgevw}F_A^*(\lambda(s))-\nabla_{\edgevw}F_A^*(\lambda(r))\|^2).
\end{align}
The quantity $\|\nabla_{\edgevw}F_A^*(\lambda(s))-\nabla_{\edgevw}F_A^*(\lambda(r))\|^2$ then needs to be controlled. 
We use the following lemma.

\begin{lemma}
For $\lambda,\lambda'\in R^{\cE}$, and $\edgevw \in E$, we have:
\begin{equation}
    \|\nabla_{\edgevw}F_A^*(\lambda)-\nabla_{\edgevw}F_A^*(\lambda')\|^2\leq 8 d_{\edgevw} \mu_{\edgevw}^2 \sum_{(\edgeuu)\sim (\edgevw)}\mu_{\edgeuu}^2 \|\lambda_{\edgeuu}-\lambda'_{\edgeuu}\|^2.
     \label{gossipnoniid1}
\end{equation}
\end{lemma}
\begin{proof}
First, notice that $\nabla_{\edgevw}F_A^*(\lambda)=\mu_{\edgevw}(\nabla g_i^*((A\lambda)_v)-\nabla g_j^*((A\lambda)_w))$. Then:
\begin{align*}
    \|\nabla f_v^*((A\lambda)_v)-\nabla f_v^*((A\lambda')_w)\| &=\|(A(\lambda-\lambda'))_v\| \text{ (smoothness)}\\
    & =\|\sum_{\edgeuu\sim \edgevw} \mu_{\edgeuu} (\lambda-\lambda')_{\edgeuu}\|\\
    & \leq  \sum_{\edgeuu\sim \edgevw} \mu_{\edgeuu} \|(x-x')_{\edgeuu}\|
\end{align*}
Conclude by taking the square and summing for $v$ and $w$.
\end{proof}

Using this with $\lambda=\lambda(s)$ and $\lambda'=\lambda(r)$:
\begin{align}
    \|\nabla_{\edgevw}F_A^*(\lambda(s))\|^2& \leq 2\|\nabla_{\edgevw}F_A^*(\lambda(r))\|^2\\
    &+2d_{\edgevw}\sum_{r<k<r+t_{\edgevw}}N((\edgevkwk),\edgevw,k)\frac{\mu_{\edgevw}^2}{2\mu_{\edgevkwk}^2}\|\nabla_{\edgevkwk}F_A^*(\lambda(k))\|^2\\
    & \leq 2\|\nabla_{\edgevw}F_A^*(\lambda(r))\|^2\\
    &+2d_{\edgevw}\sum_{r<k<r+t_{\edgevw}} \left\lceil b\frac{\ell_{\edgevw}}{L_{\edgevkwk}}\right\rceil\frac{\mu_{\edgevw}^2}{\mu_{\edgevkwk}^2}\|\nabla_{\edgevkwk}F_A^*(\lambda(k))\|^2
\end{align}
The advantage of this last expression is that only activated quantities are present on the right hand side.\\

\noindent \textbf{Step 3:} The last step of the proof consists in summing the last inequality for $t\leq \ell <t+T$, $\edgevw\in E$. When summing, each $\|\nabla_{\edgevkwk}F_A^*(\lambda(k))\|^2$ appears on the right hand-side of the inequality, with a factor upper-bounded by (here instead of $\edgevkwk$ we write $(\edgevw)$):
\begin{equation}
    2a\ell_{\edgevw}+2 d_{\edgevw}\sum_{\edgeuu\sim \edgevw}a\ell_{\edgeuu}\left\lceil\frac{b\ell_{\edgeuu}}{\ell_{\edgevw}}\right\rceil\frac{\mu_{\edgeuu}^2}{\mu_{\edgevw}^2}.
\end{equation}
We want the expression above multiplied by $C$ defined in Step 1 to be upper-bounded by $\frac{1}{4\mu_{\edgevw}^2}$, in order for \eqref{neededthm2} to be verified. This is possible if and only if:
\begin{equation}
     C\left(4a\ell_{\edgevw}\mu_{\edgevw}^2+4 d_{\edgevw}\sum_{\edgeuu\sim \edgevw}a\left\lceil\frac{b\ell_{\edgeuu}}{\ell_{\edgevw}}\right\rceil \ell_{\edgeuu}\mu_{\edgeuu}^2\right)\leq \frac{1}{2},
     \label{thm1last}
\end{equation}
where $C$ is defined in step $1$ of the proof. This is equivalent to:
\begin{align*}
      &C\left(a \ell_{\edgevw}\mu_{\edgevw}^2+d_{\edgevw}\sum_{\edgeuu\sim \edgevw}a\frac{b \ell_{\edgeuu}^2}{ \ell_{\edgevw}}\mu_{\edgeuu}^2\right)\leq \frac{1}{8}\\
      &\text{ if } \forall \edgeuu\sim \edgevw, \ell_{\edgevw}\leq b \ell_{\edgeuu},
\end{align*}
where we bounded $\left\lceil b\frac{ \ell_{\edgevw}}{ \ell_{\edgeuu}}\right\rceil$ by $2\frac{b \ell_{\edgevw}}{ \ell_{\edgeuu}}$ here. We here see that in this case, if 
\begin{equation}
    \mu_{\edgevw}^2=\frac{1}{ 2\ell_{\edgevw}}\times \min_{\edgeuu\sim \edgevw}\frac{ \ell_{\edgeuu}}{ \ell_{\edgevw}}
\end{equation}
with $8a+8d_{max}^2b\leq C^{-1}$, our inequality holds. However, our inequality on the ceil operator seems not to work in the general case. Let's take $\edgeuu$ a neighbor of $\edgevw$ such that $ \ell_{\edgevw}>b \ell_{\edgeuu}$. As $ \ell_{\edgevw}>b \ell_{\edgeuu}$, we have $\lceil\frac{b \ell_{\edgeuu}}{ \ell_{\edgevw}}\rceil=1$, leading to $a\lceil\frac{b \ell_{\edgeuu}}{ \ell_{\edgevw}}\rceil \ell_{\edgeuu}\mu_{\edgeuu}^2=a \ell_{\edgeuu}\mu_{\edgeuu}^2\leq a\leq ab$. Hence, our result still holds.\\

\noindent \textbf{Conclusion:} We have our result for $C=\frac{1}{2a+8d_{max}^2ab}$ and a laplacian weighted with local communication constraints: $\mu_{\edgevw}^2=\frac{1}{ 2\ell_{\edgevw}}\times \min_{\edgeuu\sim \edgevw}\frac{ \ell_{\edgeuu}}{ \ell_{\edgevw}}$. The final rate thus depends on the smallest eigenvalue of the laplacian weighted by:
\begin{equation}
    \frac{1}{2a+8d_{max}^2ab}\frac{1}{L_{max}}\frac{1}{ 2\ell_{\edgevw}}\times \min_{\edgeuu\sim \edgevw}\frac{ \ell_{\edgeuu}}{ \ell_{\edgevw}}\,.
\end{equation}
This ends the proof of \Cref{thm:LN_det}.
\end{proof}

\subsection{Adding stochasticity}

We now prove the following result.
\begin{theorem}[Adding Stochasticity ]\label{thm:LN_sto} Assume that, for all $k\in \N$, there exists a $\cF_{k+T-1}$-measurable event $A_k$, such that $\P(A_k|\cF_k)\geq \frac{1}{2}$ almost surely, and that under $A_k$, Assumption \ref{hyp} holds  for all $k\leq \ell \leq k+T-1$. Then, we have the following bound on $\Lambda_k(\xx)$: 
\begin{equation*}
\label{eq:lyapunov_stochasticity}
    \E[\Lambda_k(\xx)]\leq \left(\frac{1}{4}(1-\gamma)^{T/3}+\frac{3}{4}\right)^{\lceil \frac{k}{2T}\rceil}\E[\Lambda_0]\,,
\end{equation*}
where $\gamma$ is defined in \Cref{thm:LN_det}.
    
\end{theorem}

\begin{proof}
    Using the same arguments as in the proof of \Cref{thm:LN_det}, we obtain:
\begin{equation}
    \E[  \Lambda_{t+1}-  \Lambda_t|\cF_t,A_t]\leq -\sigma   \Lambda_t.
\end{equation}
However, this is not enough to conclude. Under $A_t^C$, we only know that $  \Lambda_{t+1}\leq   \Lambda_t$ (our local coordinate gradient steps cannot increase distance to the optimum). Hence:
\begin{equation}
    \E[  \Lambda_{t+1}|\cF_t]\leq (1-\sigma \mathbb{I}_{A_t})  \Lambda_t.
\end{equation}
And then, by induction:
\begin{equation}
    \E[  \Lambda_t]\leq \E[P_t\Lambda_0]  ,\text{ where } P_t=\prod_{s=0}^{t-1}(1-\sigma \mathbb{I}_{A_s}).
\end{equation}
However, no direct bound on $P_t$ exists. The interdependencies on the events $A_t$ make it impossible for an induction to prove a bound of the form $\leq (1-\sigma/2)^t$. However, the logarithm of the product seems easier to study:
\begin{equation}
    \log(P_t)=\log(1-\sigma)\sum_{s=0}^{t-1}\mathbb{I}_{A_s},
\end{equation}
giving us $\E\log(P_t)\leq\log(1-\sigma)t/2$, as $\P(A_t)\geq 1/2$. We are thus going to make a study in probability. For $t\in \N$, let $X_t=\frac{1}{T}\sum_{s=t}^{t+T-1}\mathbb{I}_{A_s}$. Using Markov-type inequalities conditionnaly on $\cF_t$ gives:
\begin{equation}
    \P(X_t\geq 1/3|\cF_t) +1/3\P(X_t\leq 1/3|\cF_t) \geq \E[ X_t|\cF_t] \geq 1/2 \implies \P(X_t\geq 1/3|\cF_t)\geq 1/4.
\end{equation}
Thus, we have: $\E[\prod_{s=t}^{t+T-1}(1-\mathbb{I}_{A_s}\sigma)|\cF_t]\leq \frac{1}{4}(1-\sigma)^{T/3}+\frac{3}{4}.$ We then know how to control $T$ consecutive factors of the product $P_t$. Skipping the next $T$ terms, we have:
\begin{align}
    \E\left[\prod_{s=t}^{t+3T-1}(1-\mathbb{I}_{A_s}\sigma)\right] & =\E\left[\prod_{s=t}^{t+T-1}(1-\mathbb{I}_{A_s}\sigma)\prod_{s=t+T}^{t+2T-1}(1-\mathbb{I}_{A_s}\sigma)\prod_{s=t+2T}^{t+3T-1}(1-\mathbb{I}_{A_s}\sigma)\right]\\
    & \leq \E\left[\prod_{s=t}^{t+T-1}(1-\mathbb{I}_{A_s}\sigma)\prod_{s=t+2T}^{t+3T-1}(1-\mathbb{I}_{A_s}\sigma)\right]\\
    & \leq \E\left[\prod_{s=t}^{t+T-1}(1-\mathbb{I}_{A_s}\sigma)\E^{\cF_{t+2T}}\left \{\prod_{s=t+2T}^{t+3T-1}(1-\mathbb{I}_{A_s}\sigma)\right\} \right]
\end{align}
as in the last right hand side, the first big product is $\cF_{t+2T}$-measurable (our asumption on the $A_s$ states that they are  $\cF_{s+T-1}$-measurable). Then, using inequality $\E\left[\prod_{s=t}^{t+T-1}(1-\mathbb{I}_{A_s}\sigma)|\cF_t\right]\leq \frac{1}{4}(1-\sigma)^{T/3}+\frac{3}{4}$ twice, with $t$ and $t+2T$, we get:
\begin{align*}
\E\left[\prod_{s=t}^{t+3T-1}(1-\mathbb{I}_{A_s}\sigma)\right]&\leq \E\left[\prod_{s=t}^{t+T-1}(1-\mathbb{I}_{A_s}\sigma) \left(\frac{1}{4}(1-\sigma)^{T/3}+\frac{3}{4}\right)\right]\\
&\leq\left(\frac{1}{4}(1-\sigma)^{T/3}+\frac{3}{4}\right)^2.\end{align*} Proceeding the same way by induction leads us to: 
\begin{equation}
    \E[P_t]\leq \left(\frac{1}{4}(1-\sigma)^{T/3}+\frac{3}{4}\right)^{\lfloor t/(2T) \rfloor},
\end{equation}
which is the desired bound. 
\end{proof}

From the proof, we thus have the following corollary.

\begin{corollary}
    Assume that, for all $k\in \N$, there exists a $\cF_{k+T-1}$-measurable event $A_k$, such that $\P(A_k|\cF_k)\geq \frac{1}{2}$ almost surely, and that under $A_k$, Assumption \ref{hyp} holds  for all $k\leq \ell \leq k+T-1$. Then, we have the following bound on $\Lambda_k(\xx)$, for any $k\geq 0$: 
\begin{equation*}
    \esp{\Lambda_{k+2T}(\xx)|\cF_k}\leq \left(\frac{1}{4}(1-\gamma)^{T/3}+\frac{3}{4}\right)\E[\Lambda_k(\xx)|\cF_k]\,.
\end{equation*}
where $\gamma$ is defined in \Cref{thm:LN_det}.

\end{corollary}

\subsection{Expliciting the constants in the loss networks model we consider}

We now need to compute and tune the constants introduced in \Cref{thm:LN_det} for the assumptions of \Cref{thm:LN_sto} to hold in our Loss Network model.
We begin by the following lemma, inspired by queuing theory arguments, that upper bound the probability that an edge stays inactivated for a long period of time.

Note that we here come back to continuous time, to study the loss network model. What is important to keep in mind is that an edge cannot be occupied for a time longer than $\tau'_\edgevw$.

\begin{lemma}\label{lem_queue_1}Let $\delta\in(0,1)$.
For any $t_0\geq 0$, $\edgevw\in E$, if the \emph{Poisson} intensities are such that $p_{\edgevw}=\frac{1}{2\max(d_i,d_j)-1}(  \tau'_{\edgevw})^{-1}$ and $ \tau'_{max}(\edgevw)=\max_{\edgeuu\sim \edgevw} \tau'_{\edgeuu}$, let:
\begin{equation*}
    \ell_{\edgevw}=\frac{\log(\delta^{-1})}{\log(1-(1-e^{-1})e^{-1})}(p_{\edgevw}^{-1}+ \tau'_{max}(\edgevw)) \,.
\end{equation*} We have:
\begin{equation}
    \P(\edgevw\text{ not activated in $[t_0,t_0+\ell_{\edgevw}]$}|\F_{t_0})\leq \delta.
\end{equation}
\end{lemma}
\begin{proof}[Proof of Lemma \ref{lem_queue_1}]Let $\edgevw\in E$ and $t_0\geq 0$ fixed. We use tools from queuing theory \cite[$M/M/\infty/\infty$ queues]{Tanner1995queuing} in order to compute the probability that edge $\edgevw$ is activable at a time $t$ or not. More formally, we define a process $N_{\edgevw}(t)$ with values in $\N$, such that $N_{\edgevw}(t_0)=1$ if $\edgevw$ non-available at time $t_0$ and $0$ otherwise. Then, when an edge $\edgeuu$ such that $\edgeuu\sim \edgevw$ is activated, we make an increment of $1$ on $N_{\edgevw}(t)$ (a \emph{customer} arrives). This customer stays for a time $ \tau'_{\edgeuu} $ and when he leaves, $N_{\edgevw}$ is decreased by $1$. Thus $N_{\edgevw}\geq 0$ a.s., and if $N_{\edgevw}=0$, then edge $\edgevw$ is available. For $t\geq \max_{\edgeuu\sim \edgevw}  \tau'_{\edgeuu}  +t_0$, $N_{\edgevw}(t)$ follows a Poisson law of parameter $\sum_{\edgeuu\sim \edgevw}p_{\edgeuu} \tau'_{\edgeuu} $.  For any $t\geq \max_{\edgeuu\sim \edgevw}  \tau'_{\edgeuu}  +t_0$:
\begin{equation*}
    \P(\edgevw\text{ available at time }t|\F_{t_0})\geq \P(N_i(t)= 0)=\exp(-\sum_{\edgeuu\sim \edgevw}p_{\edgeuu} \tau'_{\edgeuu} ).
\end{equation*}
That leads to taking $p_{\edgeuu}=\frac{1}{2}\frac{1}{\max(d_k,d_l)-1}(  \tau'_{\edgeuu})^{-1}$ for all edges, in order to have $$\P(\edgevw\text{ available at time }t|\F_{t_0})\geq 1/e.$$ Then, $\P(\edgevw \text{ rings in } [t,t+p_{\edgevw}^{-1}])=1-e^{-1}$, giving:
\begin{align*}
    \P&(\edgevw \text{ activated in }[t_0,t_0+  \tau'_{\max}(\edgevw)+p_{\edgevw}^{-1}]|\F_{t_0})= \P(\edgevw \text{ rings in } [t,t+p_{\edgevw}^{-1}])\\
    &\times \P(\edgevw \text{ available at time }t|\F_{t_0},\text{$\edgevw$ rings at a time } t\in [t_0+  \tau'_{\max}(\edgevw),t_0+  \tau'_{\max}(\edgevw)+p_{\edgevw}^{-1}])\\
    & \geq (1-e^{-1})e^{-1},
\end{align*}
where we use the memoriless property of exponential random variables. Take $k\in \N$ such that $(1-(1-e^{-1})e^{-1})^k\leq \delta$, leading to $k= \log(6|E|)/\log(1-(1-e^{-1})e^{-1})$. Let $$\ell_{\edgevw}=k(p_{\edgevw}^{-1}+ \tau'_{max}(\edgevw) ).$$ Then we have a.s.:
\begin{equation}
    \P(\edgevw\text{ not activated in $[t_0,t_0+\ell_{\edgevw}]$}|\F_{t_0})\leq \delta.
\end{equation}
\end{proof}

Let $t\in \N$ be fixed, and $B_t$ be the event: "in the activations $t,t+1,...,t+T-1$, all edges are activated". Let then $C_t(\edgevw,s)$ for $t\leq s < t+T$ be the event $\min(T_{\edgevw}(s),t+T-s,s-t)\leq a\ell_{\edgevw}$ and $D_t(\edgeuu,\edgevw,s)$ be the event $N(\edgeuu,\edgevw,s)\leq \lceil b\ell_{\edgevw}/\ell_{\edgeuu} \rceil$, where $N(\edgeuu,\edgevw,s)$ is the number of activations of $\edgeuu$ between two activations of $\edgevw$, around time $s$, where we only take into account the activations between activations $t$ and $t+T-1$. Let then $A_t=B_t\cap (\cap_{\edgeuu,\edgevw\in E, t\leq s < t+T}C_t(\edgevw,s)\cap D_t(\edgeuu,\edgevw,s))$. 

We want $\P(A_t)\geq 1/2$ for correct constants $a,b,T$ and $\ell_{\edgevw}$ (that can differ from $\tau'_{\edgevw}$) in order to apply \Cref{thm:LN_det,thm:LN_sto}. Note that this event is $\F_{t+T-1}$-measurable, as desired. We first study the length of time $\ell_{\edgevw}$ edge $\edgevw$ must wait in order to be activated with high probability (\emph{high} meaning more that $1-\frac{1}{12|E|}$). This result is Lemma \ref{lem_queue_1}. Then, we use this length to determine the constants $T,a,b,\ell_{\edgevw}$ needed.\\

\begin{lemma}
For any continuous time $t_0\geq 0$, $\edgevw\in \cE$, if $p_{\edgevw}=\frac{1}{2\max(d_i,d_j)-1}(\tau'_{\edgevw})^{-1}$ and $\tau'_{max}(\edgevw)=\max_{\edgeuu\sim \edgevw}\tau'_{\edgeuu}$, let $\ell_{\edgevw}=\frac{\log(6|E|)}{\log(1-(1-e^{-1})e^{-1})}(p_{\edgevw}^{-1}+\tau'_{max}(\edgevw))$. We have, almost surely:
\begin{equation}
    \P(\edgevw\text{ not activated in $[t_0,t_0+\ell_{\edgevw}]$}|\F_{t_0})\leq \frac{1}{6|E|}.
\end{equation}
\end{lemma}
\begin{proof}[Proof of Lemma \ref{lem_queue_1}]Let $\edgevw\in E$ and $t_0\geq 0$ fixed. We use tools from queuing theory \citep{Tanner1995queuing} ($M/M/\infty/\infty$ queues) in order to compute the probability that edge $\edgevw$ is activable at a time $t$ or not. More formally, we define a process $N_{\edgevw}(t)$ with values in $\N$, such that $N_{\edgevw}(t_0)=1$ if $\edgevw$ non-available at time $t_0$ and $0$ otherwise. Then, when an edge $\edgeuu,\edgeuu\sim \edgevw$ is activated, we make an increment of $1$ on $N_{\edgevw}(t)$ (a \emph{customer} arrives). This customer stays for a time $\tau'_{\edgeuu}$ and when he leaves we make $N_{\edgevw}$ decrease by $1$. We have $N_{\edgevw}\geq 0$ a.s., and if $N_{\edgevw}=0$, $\edgevw$ is available. For $t\geq \max_{\edgeuu\sim \edgevw} \tau'_{\edgeuu} +t_0$, $N_{\edgevw}(t)$ follows a Poisson law of parameter $\sum_{\edgeuu\sim \edgevw}p_{\edgeuu}\tau'_{\edgeuu}$.  For any $t\geq \max_{\edgeuu\sim \edgevw} \tau'_{\edgeuu} +t_0$:
\begin{equation}
    \P(\edgevw\text{ available at time }t|\F_{t_0})\geq \P(N_i(t)= 0)=\exp(-\sum_{\edgeuu\sim \edgevw}p_{\edgeuu}\tau'_{\edgeuu}).
\end{equation}
That leads to taking $p_{\edgeuu}=\frac{1}{2}\frac{1}{\max(d_k,d_l)-1}(\tau'_{\edgeuu})^{-1}$ for all edges, in order to have $\P(\edgevw\text{ available at time }t|\F_{t_0})\geq 1/e$. Then, $\P(\edgevw \text{ rings in } [t,t+p_{\edgevw}^{-1}])=1-e^{-1}$, giving:
\begin{align}
    \P(\edgevw &\text{ activated in }[t_0,t_0+\tau'_{\max}(\edgevw)+p_{\edgevw}^{-1}]|\F_{t_0})= \P(\edgevw \text{ rings in } [t,t+p_{\edgevw}^{-1}])\\
    &\times \P(\edgevw \text{ available at time }t|\F_{t_0},\text{$\edgevw$ rings at a time}\\
    & t\in [t_0+\tau'_{\max}(\edgevw),t_0+\tau'_{\max}(\edgevw)+p_{\edgevw}^{-1}])\\
    & \geq (1-e^{-1})e^{-1},
\end{align}
where we use the fact that exponential random variables have no memory. Take $k\in \N$ such that $(1-(1-e^{-1})e^{-1})^k\leq \frac{1}{6|E|}$, leading to $k\approx \log(6|E|)/\log(1-(1-e^{-1})e^{-1})$. Let $\ell_{\edgevw}=k(p_{\edgevw}^{-1}+\tau'_{max}(\edgevw))$. Then we have a.s.:
\begin{equation}
    \P(\edgevw\text{ not activated in $[t_0,t_0+\ell_{\edgevw}]$}|\F_{t_0})\leq \frac{1}{6|E|}.
\end{equation}
\end{proof}

\noindent \textbf{Bounding $T$:} A direct application of \Cref{lem_queue_1} leads, with $L=\max_{\edgevw}\ell_{\edgevw}$, to:
\begin{equation}
    T=2\sum_{\edgevw}\frac{L}{\tau'_{\edgevw}}.
\end{equation}
Indeed, for all $\edgevw$, not being activated in activations $t,t+1,...,t+T-1$ means not being activated for a continuous interval of time of length more than $\ell_{\edgevw}$. Hence:
\begin{align}
    &\P(\exists (\edgevw)\in E: (\edgevw)\text{ not activated in }\{t,...,t+T-1\}|\F_t)\\
    &\leq \sum_{\edgevw\in E}\P((\edgevw)\text{ not activated in }\{t,...,t+T-1\}|\F_t)\\
    &\leq \sum_{\edgevw\in E}\P((\edgevw)\text{ not activated in }[t,t+\ell_{\edgevw}]|\F_t)\\
    &\leq |E|\times \frac{1}{6|E|}\\
    &=1/6. \label{eq:B_t}
\end{align}

\noindent \textbf{Bounding $T_{\edgevw}$:} Applying Lemma \ref{lem_queue_1} with $12|E|T$ instead of $6|E|$ leads to controlling all the inactivation lengths by a length $\ell'_{\edgevw}$, with a probability more than $1-1/(12|E|T)$. Let $\edgevw\in E$ and $s\in \N$, $t\leq s < t+T$. Let $\alpha>0$ to tune later. Denote by $\delta_{\edgevw}(s)$ the (random) inactivation time of $\edgevw$, around iteration $s$. Note that conditionnaly on the inactivation period $\delta_{\edgevw}(s)$, $T_{\edgevw}(s)$ is dominated in law by a Poisson variable of parameter $I\delta_{\edgevw}(s)$, hence line \eqref{eq:poissondelta}:
\begin{align}
    \P(T_{\edgevw}(s)\geq \alpha \ell_{\edgevw}'|\F_t)&\leq\P(T_{\edgevw}(s)\geq \alpha \ell_{\edgevw}'|\F_t, \delta_{\edgevw}\leq \ell_{\edgevw}')\times \P(\delta_{\edgevw}\leq \ell_{\edgevw}') + \P(\delta_{\edgevw}\geq \ell_{\edgevw}')\\
    &\leq \P(Poisson(I\ell'_{\edgevw})\geq \alpha \ell_{\edgevw}') + \frac{1}{12|E|T} \quad (\text{where} \quad I=\sum_{\edgevw\in\cE}p_\edgevw)\label{eq:poissondelta}\\
    &\leq \frac{1}{12|E|T}+\frac{1}{12|E|T}\label{eq:wantedpoisson}\\
    &=\frac{1}{6|E|T},
\end{align}
for some $\alpha>0$ big enough, to determine with the following large deviation inequality:
\begin{lemma}[A Large Deviation Inequality on discrete Poisson variables.] \label{poisson}Let $Z\sim Poisson(\lambda)$, for some $\lambda>0$. Then, for all $u\geq 0$:
\begin{equation}
    \P(Z\geq u)\leq \exp(-u+\lambda(e-1)).
\end{equation}
\end{lemma}
\noindent This large deviation leads to taking $\alpha=2eI$ for \eqref{eq:wantedpoisson} to be true. Finally, we get:
\begin{equation}
    \P(T_{\edgevw}(s)\geq \alpha \ell_{\edgevw}'|\F_t)\leq \frac{1}{6|E|T}. \label{eq:C_t}
\end{equation}

\noindent \textbf{Bounding $N(\edgeuu,\edgevw,s)$:} If $\delta_{\edgevw}(s)\leq \ell'_{\edgevw}$, this random variable is dominated by a Poisson variable of parameter $p_{\edgeuu}\ell_{\edgevw}'$. Hence, still with Lemma \ref{poisson}, with probability more than $1-\frac{1}{12|E|^2T}$, we can bound $N(\edgeuu,\edgevw)$ by $e\log(12|E|^2T)+p_{\edgeuu}\ell_{\edgevw}(e-1)\leq 2 e p_{\edgeuu}L_{\edgevw}$.\\

\noindent \textbf{Explicit writing of the union bound on $A_t^C$:} $A_t^C=B_t^C\cup (\cup_{\edgeuu,\edgevw\in E, t\leq s < t+T}C_t(\edgevw,s)^C\cup D_t(\edgeuu,\edgevw,s)^C)\in \F_{t+T-1}$. Thanks to the previous considerations, we have that $\P^{\F_t}(B_t^C) \leq 1/6$ with \eqref{eq:B_t}, $\P^{\F_t}(C_t(\edgevw,s)^C)\leq \frac{1}{6|E|T}$ with \eqref{eq:C_t} and $\P(D_t(\edgeuu,\edgevw,s)^C|\F_t)\leq \frac{1}{6|E|^2T}$, for the following constants and weights:
\begin{itemize}
    \item $\Tilde{\tau'}_{\edgevw}^{-1}=p_{\edgevw}=\min(\frac{1}{\tau'_{\max}(\edgevw)},\frac{1}{2(\max(d_i,d_j)-1)}\frac{1}{\tau'_{\edgevw}})$;
    \item $T=2 I \max_{\edgevw\in E}\Tilde{\tau'_{\edgevw}}\frac{\log(6|E|)}{\log(1-(1-e^{-1})e^{-1})}$;
    \item $a=2eI\frac{\log(6|E|T)}{\log(1-(1-e^{-1})e^{-1})}$;
    \item $b=2e\frac{\log(6|E|T)}{\log(1-(1-e^{-1})e^{-1})}$.
\end{itemize}
The union bound is the following:
\begin{align}
    \P^{\F_t}(A_t^C) & \leq  \P^{\F_t}(B_t^C) +\sum_{s,\edgevw}\P^{\F_t}(C_t(\edgevw,s)^C) +\sum_{s,\edgevw}\P^{\F_t}(\cup_{\edgeuu}D_t(\edgeuu,\edgevw,s)^C)\\
    & \leq 1/6 + |E|T/(6|E|T)\times 2\\
    & \leq 1/2.
\end{align}

\noindent The rate of convergence $\gamma$ is then defined as the smallest non null eigenvalue of the laplacian of the graph, weighted by:
\begin{equation}\label{eq:rho_ln}
   \nu_{\edgevw}= \frac{p_{\edgevw}\min_{\edgeuu\sim \edgevw}\frac{\tau'_{\edgevw}}{\tau_{\edgeuu}}}{8a(1+d^2b)}=\frac{\min_{\edgeuu\sim \edgevw}p_{\edgeuu}}{c_1 \ln(6|\cE|T) (1+d^2 \ln(6|\cE|T)^2)\sum_{\edgeuu\in\cE}p_\edgeuu}
\end{equation}

\subsection{Concluding}

What we have proved so far, is that for any $k\geq 0$, any $\xx\in\R^\cV$, we have:
\begin{equation*}
    \esp{\Lambda_{k+2T}(\xx)|\cF_k}\leq \left(\frac{1}{4}(1-\gamma)^{T/3}+\frac{3}{4}\right)\E[\Lambda_k(\xx)|\cF_k]\,,
\end{equation*}
where $\gamma$ is defined in \Cref{eq:rho_ln}. Then, $\Lambda_{k+2T}(\xx)\geq \frac{1}{2}\NRM{W^{(0,k+2T)}(\xx-\bar\xx)}^2$ and $\Lambda_k(\xx)\leq \frac{1}{2}\NRM{W^{(0,k)}(\xx-\bar\xx)}^2$, so that applying this for $k=0$, almost surely conditionned on $\cF_0$,
\begin{equation*}
    \esp{\NRM{W^{(0,2T)}(\xx-\bar\xx)}^2|\cF_0}\leq \left(\frac{1}{4}(1-\gamma)^{T/3}+\frac{3}{4}\right)\E[\NRM{\xx-\bar\xx}^2|\cF_0]\,,
\end{equation*}
Now, noticing that our analysis holds almost surely for any configuration $\cF_0$, doing a time translation and starting from a configuration $\cF_k$ for any $k$, we get that:
\begin{equation*}
    \esp{\NRM{W^{(k,k+2T)}(\xx-\bar\xx)}^2|\cF_k}\leq \left(\frac{1}{4}(1-\gamma)^{T/3}+\frac{3}{4}\right)\E[\NRM{\xx-\bar\xx}^2|\cF_k]\,,
\end{equation*}
so that \Cref{hyp:consensus} holds for $\rho=\frac{1}{4}(1-(1-\gamma)^{T/3}),k_\rho=2T$, and hence $\frac{\rho}{k_\rho}=\cO(\gamma)$, which leads to \Cref{thm:LN}:
$\gamma$ is the eigengap of the graph, with weights of order $\tilde\cO(\frac{\min_{\edgeuu\sim\edgevw}p_\edgeuu}{d^2\sum_{\edgeuu\in\cE}p_\edgeuu})$.

\section{Proof of Theorem~\ref{thm:lip-conv}: Convex-Lipchitz case}

\subsection{Homogeneous setting, Lipschitz (bounded gradients) and convex without sampling}

\begin{proof}
    Studying the virtual sequence, we expand:
    \begin{align*}
        \esp{\NRM{\hat x^{k+1}- x^\star}^2} &= \esp{ \NRM{\hat x^k- x^\star}^2 - \frac{2\gamma}{n} \sum_{v\in\cI_k}\langle \nabla f_v( x_{v}^k),\hat x^k- x^\star\rangle + \frac{\gamma^2}{n^2}\NRM{\sum_{v\in\cI_k} g^k_v}^2 }\\
        &\leq \esp{ \NRM{\hat  x^k- x^\star}^2 - \frac{2\gamma}{n} \sum_{v\in\cI_k}\langle \nabla f_v( x_{v}^k), x_{v}^k- x^\star\rangle + \frac{2\gamma}{n} \sum_{v\in\cI_k}\langle \nabla f_v( x_{v}^k), x_{v}^k-\bar x^k\rangle + \frac{2\gamma}{n} \sum_{v\in\cI_k} \langle \nabla f_v( x_{v}^k),\bar x^k-\hat x^k\rangle }\\
        &\quad + \frac{\gamma^2B^2\nk^2}{n^2}\,,
    \end{align*}
    where we used the Lipschitz assumption, $\E \gg_v^k = \nabla f_v(x_v^k)$ and boundness of gradients.
    Denote:
    \begin{align*}
        T_1&= - \frac{2\gamma}{n} \sum_{v\in\cI_k} \langle \nabla f_v( x_{v}^k), x_{v}^k- x^\star\rangle \\
        T_2^k&= \frac{2\gamma}{n}\sum_{v\in\cI_k} \langle \nabla f_v( x_{v}^k), x_{v}^k-\bar x^k\rangle \\
        T_3&= \frac{2\gamma}{n}\sum_{v\in\cI_k} \langle \nabla f_v( x_{v}^k),\bar x^k-\hat x^k\rangle \,.\\
    \end{align*}
        Using convexity of $f$, 
    \begin{equation*}
        T_1\leq -\frac{2\gamma}{n} \sum_{v\in\cI_k} (f( x_{v}^k) -f( x^\star))\,.
    \end{equation*}
    Using the Lipschitz assumption and Equation~\eqref{eq:bound_virtual_B} that controls $\NRM{\bar x^k-\hat x^k}$, we bound $T_3$::
    \begin{align*}
        T_3\leq \frac{2\gamma^2B^2\nk}{n}\,.
    \end{align*}
    Using the Lipschitz assumption and our consensus bound from Equation~\eqref{eq:bound_consensus_B}, we bound $T_2^k$:
    \begin{align*}
        \sum_{k<K}T_2^k&\leq \sum_{k<K}\frac{2\gamma B}{n}\sqrt{\sum_{v\in\cI_k}\esp{\NRM{ x_{v}^k-\bar x^k}^2}}\\
        &\leq\sum_{k<K}\frac{2\gamma B}{n}\sqrt{\esp{\NRM{ \xx^k-\bar \xx^k}^2}}\\
        &\leq \sum_{k<K}\frac{\gamma^2B^2}{n\bar\rho} + \frac{\bar\rho}{B}\esp{\NRM{ \xx^k-\bar \xx^k}^2} \\
        &\leq \frac{3\gamma^2B^2}{n\bar\rho}\NK\,.
    \end{align*}
    Consequently, denoting $\eta=\frac{\gamma}{n}$ and summing over $k<K$,
    \begin{align*}
        2\eta\sum_{k<K}\sum_{v\in\cI_k}\esp{f( x_{v}^k)-f( x^\star)}&\leq \esp{ \NRM{\hat  x^0- x^\star}^2} + \eta^2B^2\left( \E\nk + 2n + 3n\bar\rho^{-1}\right)\NK\\
        &\leq \esp{ \NRM{\hat  x^0- x^\star}^2} + \eta^2B^2\left( 3n + 3n\bar\rho^{-1}\right)\NK\,.
    \end{align*}
    Dividing by $2\eta\NK$,
    \begin{equation*}
        \esp{\frac{1}{\NK}\sum_{k<K}\sum_{v\in\cI_k}f\left( x_{v}^k\right)-f( x_\star)}\leq \frac{\esp{\NRM{\hat  x^0- x^\star}^2}}{2\eta \NK} + \frac{\eta B^2}{2}( 3n + 3n\bar\rho^{-1})\,,
    \end{equation*}
    and 
    \begin{align*}
        \esp{\NRM{\hat  x^0- x^\star}^2}&\leq \NRM{ x^0- x^\star}^2 -2\eta\sum_{v\in\cV}\langle \nabla f( x^0), x^0- x^\star\rangle + \eta^2G^2/n\\
        &\leq \NRM{ x^0- x^\star}^2 + \eta^2B^2/K\,,
    \end{align*}
    provided that $K\geq n$.
    Optimizing over $\eta$, we obtain that for $\eta=\sqrt{\frac{D^2}{2KB^2( 3n + 2n\bar\rho^{-1})}}$,
    \begin{equation*}
        \esp{f\left(\frac{1}{\NK}\sum_{k=0}^{K-1} \sum_{v\in\cI_k}x_{v}^k\right)-f( x_\star)}\leq 2\sqrt{\frac{2B^2D^2( 3n + 2n\bar\rho^{-1})}{\NK}}\,.
    \end{equation*}
    \end{proof}

\subsection{Lipschitz (bounded gradients) and convex with sampling}

\begin{proof}
    Taking the proof just above, we still have
    \begin{align*}
        \esp{\NRM{\hat  x^{k+1}- x^\star}^2} \leq \esp{ \NRM{\hat  x^k- x^\star}^2 +T_1^k+T_2^k+T_3 } + \frac{\gamma^2B^2\nk^2}{n^2}\,.
    \end{align*}
    We have, using convexity and then Lipschitzness:
    \begin{align*}
        T_1^k&=- \frac{2\gamma}{n} \sum_{v\in\cI_k} \langle \nabla f_v( x_{v}^k), x_{v}^k- x^\star\rangle\\
        &\leq -\frac{2\gamma}{n}  \sum_{v\in\cI_k}p_v f_v( x_{v}^k) -f( x^\star)\\
        & =  -\frac{2\gamma}{n}  \sum_{v\in\cI_k}p_v f_v( \bar x^k) -f( x^\star) + f_v( x_{v}^k) -f( \bar x^k)\\
        & \leq  -\frac{2\gamma}{n}  \sum_{v\in\cI_k}f_v( \bar x^k) -f( x^\star) - B\NRM{ x_{v}^k-\bar x^k}\,,
    \end{align*}
    so that
    \begin{align*}
        \esp{T_1^k}&\leq -\frac{2\gamma\bar p }{n} (\E f( \bar x^k) -f( x^\star)) + \frac{2\gamma B}{n}\sum_{v\in\cV}p_v\NRM{x_{v}^k-\bar x^k}\\
        & \leq -\frac{2\gamma\bar p }{n} (\E f( \bar x^k) -f( x^\star)) + \frac{2\gamma B p_{\max}}{n}\sqrt{n}\NRM{\xx^k-\bar\xx^k}\,.
    \end{align*}
    We then have that:
    \begin{equation*}
        \sum_{k<K}\frac{2\gamma B p_{\max}}{n}\sqrt{n}\NRM{\xx^k-\bar\xx^k} \leq 2\sqrt{2} \frac{\gamma^2B^2p_{\max}}{\sqrt{n}}\sqrt{K\NK}\,.
    \end{equation*}
    Then,
    \begin{align*}
        T_3 \leq \frac{2\gamma^2B^2}{n}\,.
    \end{align*}
    We handle the consensus term differently. For some $\alpha>0$ to be fix later, and taking the expectation conditionnally on $\xx^k$,
    \begin{align*}
        \sum_{k<K}\esp{T_2^k}&\leq \sum_{k<K}\sum_{v\in\cI_k}\esp{\frac{\gamma^2}{n\alpha}\NRM{\nabla f(x_{v}^k)}^2+ \frac{\alpha}{n}\NRM{x_{v}^k-\bar x^k}^2}\\
        &\leq\sum_{k<K}\frac{\gamma^2B^2\nk}{\alpha n}+ \frac{\alpha}{n}\sum_{v\in\cV}p_v\NRM{x_{v}^k-\bar x^k}^2\\
        &\leq\frac{\gamma^2B^2}{\alpha n}\NK+ \frac{\alpha p_{\max}}{n}\sum_{k<K}\E\NRM{\xx^k-\bar \xx^k}^2\\
        &\leq \left(\frac{\gamma^2B^2}{\alpha n} + \frac{\alpha p_{\max}}{n}\frac{2\gamma^2B^2}{\bar\rho^2}\right)\NK\,.
    \end{align*}
    We set $\alpha=1/\sqrt{p_{\max}\bar\rho^{-2}}$, so that:
    \begin{equation*}
        \sum_{k<K}\esp{T_2^k}\leq 2\frac{\gamma^2B^2}{n^2} \times \sqrt{p_{\max}}n\bar\rho^{-1}\times \NK\,.
    \end{equation*}
    The rest of the proof then follows as before, and we obtain
        \begin{equation*}
            \esp{f\left(\frac{1}{\NK}\sum_{k=0}^{K-1}\sum_{v\in\cI_k} x_{v}^k\right)-f( x^\star)}=\cO\left(\sqrt{\frac{B^2D^2}{\NK}(n+(p_{\max})^{1/2}n\bar\rho^{-1} + n^{3/2}p_{\max} \sqrt{\frac{K}{\NK}} )}\right)\,.
        \end{equation*}
        To conclude, we notice that $\frac{nK}{\NK}$ is of order $1/\bar p$ where $\bar p= \frac{1}{n}\sum_{v\in\cV}$.
\end{proof}

\section{Proof of \Cref{thm:lip-smooth-conv}: smooth-Lipschitz-convex rates}

\subsection{Smooth-Lipschitz-convex rates without sampling, homogeneous case}

\begin{proof}
    As before, we have:
    \begin{align*}
        \esp{\NRM{\hat  x^{k+1}- x^\star}^2} \leq \esp{ \NRM{\hat  x^k- x^\star}^2+ T_1+T_2^k+T_3 }  + \frac{\gamma^2\sigma^2\nk + \gamma^2\E\NRM{\sum_{v\in\cI_k}\nabla f_v(x_{v}^{k})}^2}{n^2}\,,
    \end{align*}
    with
    \begin{align*}
        T_1&= - \frac{2\gamma}{n} \sum_{v\in\cI_k} \langle \nabla f_v( x_{v}^k), x_{v}^k- x^\star\rangle \\
        T_2^k&= \frac{2\gamma}{n}\sum_{v\in\cI_k} \langle \nabla f_v( x_{v}^k), x_{v}^k-\bar x^k\rangle \\
        T_3&= \frac{2\gamma}{n}\sum_{v\in\cI_k} \langle \nabla f_v( x_{v}^k),\bar x^k-\hat x^k\rangle \,.\\
    \end{align*}
    First, using convexity of $f_v\equiv f$, 
    \begin{equation*}
        T_1\leq -\frac{2\gamma}{n} \sum_{v\in\cI_k} (f_v( x_{v}^k) -f_v( x^\star))=-\frac{2\gamma}{n} \sum_{v\in\cI_k} (f( x_{v}^k) -f( x^\star))\,.
    \end{equation*}
    Using Assumption~\ref{hyp:consensus2} we have, where $C>0$ can be arbitrary:
    \begin{align*}
        \esp{T_2^k}&\leq \frac{2\gamma}{n} \sum_{v\in\cI_k} \esp{\NRM{\nabla f( x_{v}^k)}\NRM{ x_{v}^k-\bar x^k}}\\
        &\leq \frac{C\gamma}{n}\sum_{v\in\cI_k}\esp{\NRM{\nabla f( x_{v}^k)}^2} + \frac{\gamma}{Cn}\esp{\sum_{v\in\cI_k}\NRM{ x_{v_k}^k-\bar x^k}^2}\\
        &\leq \frac{2LC\gamma}{n}\sum_{v\in\cI_k}\esp{(f( x_{v}^k)-f( x^\star))} + \frac{\gamma}{Cn}\esp{\NRM{ \xx^k-\bar \xx^k}^2}\,. 
    \end{align*}    
    We also have:
    \begin{align*}
        T_3 &\leq \frac{\gamma}{n}\Big( C\sum_{v\in\cI_k}\NRM{\nabla f( x_{v}^k)}^2 +\frac{1}{C}\NRM{\bar x^k-\hat x^k}^2 \Big)\\
        &\leq \frac{\gamma}{n}\Big( 2LC\sum_{v\in\cI_k}(f( x_{v}^k)-f( x^\star)) +\frac{\gamma^2B^2}{C}\nk \Big)\,.
    \end{align*}
    Thus,
    \begin{align*}
        \frac{2\gamma}{n} \sum_{v\in\cI_k} (\E f( x_{v}^k) -f( x^\star)) &\leq -\esp{\NRM{\hat  x^{k+1}- x^\star}^2} +\esp{ \NRM{\hat  x^k- x^\star}^2} + \frac{\gamma^2\sigma^2\nk}{n^2} + \frac{2\gamma^2 L\nk}{n^2} \sum_{v\in\cI_k} (\E f( x_{v}^k) -f( x^\star)) \\
        & \quad + \frac{2LC\gamma}{n}\sum_{v\in\cI_k}\esp{(f( x_{v}^k)-f( x^\star))} + \frac{\gamma}{Cn}\esp{\NRM{ \xx^k-\bar \xx^k}^2}\\
        &\quad + \frac{\gamma}{n}\Big( 2LC\sum_{v\in\cI_k}(f( x_{v}^k)-f( x^\star)) +\frac{\gamma^2B^2}{C} \nk\Big)\,.
    \end{align*}
    Summing over $k<K$ and using \Cref{lem:consensus_control}, we obtain:
    \begin{align*}
        \frac{2\gamma}{n} \sum_{k<K}\sum_{v\in\cI_k} (\E f( x_{v}^k) -f( x^\star)) &\leq \esp{ \NRM{\hat  x^0- x^\star}^2} + \frac{\gamma^2\sigma^2}{n^2}\NK + \frac{2\gamma L}{n}\big( 2C+\gamma \big) \sum_{k<K}\sum_{v\in\cI_k} (\E f( x_{v}^k) -f( x^\star)) \\
        & \quad + \sum_{k<K}\frac{\gamma}{Cn}\esp{\NRM{ \xx^k-\bar \xx^k}^2}+\frac{\gamma^3B^2}{Cn}\NK\\
        &\leq \esp{ \NRM{\hat  x^0- x^\star}^2} + \frac{\gamma^2\sigma^2}{n^2}\big(1+\frac{2\gamma \bar\rho^{-1}n}{C}\big)\NK+ \sum_{k<K}\frac{\gamma}{Cn}\esp{\NRM{ \xx^k-\bar \xx^k}^2}+\frac{\gamma^3B^2}{Cn}\NK\\
        &\quad + \frac{2\gamma L}{n}\big( 2C+\gamma + \frac{4\gamma^2}{\bar\rho^{2}} \big) \sum_{k<K}\sum_{v\in\cI_k} (\E f( x_{v}^k) -f( x^\star)) \,.
    \end{align*}
    Hence, provided that $ 2C+\gamma + \frac{4\gamma^2}{\bar\rho^{2}} \leq \frac{1}{2L}$, which is verified for $C=\frac{1}{8L}$ and $\gamma\leq \frac{1}{4L}\times \frac{1}{1+2\bar\rho^{-1}}$, we have:
    \begin{align*}
        \frac{\gamma}{n} \sum_{k<K}\sum_{v\in\cI_k} (\E f( x_{v}^k) -f( x^\star)) \leq \esp{ \NRM{\hat  x^0- x^\star}^2} + \frac{\gamma^2\sigma^2}{n^2}\big(1+16L\gamma \bar\rho^{-1}n\big)\NK + \frac{8L\gamma^3B^2}{n}\NK\,,
    \end{align*}
    leading to, for $\eta=\gamma/n$:
    \begin{align*}
        \esp{f\left(\frac{1}{\NK}\sum_{k=0}^{K-1}\sum_{v\in\cI_k} x_{v}^k\right)-f( x^\star)} \leq \frac{\esp{ \NRM{\hat  x^0- x^\star}^2}}{\eta\NK} + \eta \sigma^2 + \eta^2\left(16L\sigma^2n^2\brhom + 8LB^2n^2\right) \,.
    \end{align*}
    Optimizing over $\eta\leq \frac{1}{4L}\times \frac{1}{n(1+2\bar\rho^{-1})}$, we thus obtain that:
    \begin{equation*}
        \esp{f\left(\frac{1}{\NK}\sum_{k=0}^{K-1}\sum_{v\in\cI_k} x_{v}^k\right)-f( x^\star)} = \cO \left( \frac{LD^2 n\brhom}{\NK} + \sqrt{\frac{D\sigma^2}{\NK}} + \left[  \frac{D^2\sqrt{LB^2n^2 + L\sigma^2n^2\brhom }}{\NK} \right]^{2/3} \right)\,.
    \end{equation*}
    
\end{proof}

\subsection{Smooth-Lipschitz-convex rates with sampling, heterogeneous case} \label{sec:het-lipshitz-convex}

\begin{proof}
    We have:
    \begin{equation*}
        \esp{\NRM{\hat  x^{k+1}- x^\star}^2} \leq \esp{ \NRM{\hat  x^k- x^\star}^2 - \frac{2\gamma}{n}\sum_{v\in\cI_k}\langle \nabla f_v(x_v^k),\hat x^k-x^\star\rangle }  + \frac{\gamma^2\sigma^2\nk + \gamma^2\E\NRM{\sum_{v\in\cI_k}\nabla f_v(x_{v}^{k})}^2}{n^2}\,,
    \end{equation*}
    and we will handle the middle term differently than before.
    Using $- \frac{2\gamma}{n}\sum_{v\in\cI_k}\langle \nabla f_v(x_v^k),\hat x^k-x^\star\rangle = - \frac{2\gamma}{n}\sum_{v\in\cI_k}\langle \nabla f_v(x_v^k),  x_v^k-x^\star\rangle - \frac{2\gamma}{n}\sum_{v\in\cI_k}\langle \nabla f_v(x_v^k),\hat x^k-x_v^k \rangle $ and then convexity for the first term and smoothness for the second, we obtain:
    \begin{align*}
        - \frac{2\gamma}{n}\sum_{v\in\cI_k}\langle \nabla f_v(x_v^k),\hat x^k-x^\star\rangle &\leq - \frac{2\gamma}{n}\sum_{v\in\cI_k}\left(f_v(x_v^k)-f_v(x^\star) - \frac{2\gamma}{n}\sum_{v\in\cI_k}f_v(\hat x^k) -f_v(x_v^k) -\frac{L}{2}\NRM{x_v^k - \hat x^k}^2\right)\\
        &=- \frac{2\gamma}{n}\sum_{v\in\cI_k} f_v(\hat x^k) -f_v(x^\star)   + \frac{\gamma L}{n}\sum_{v\in\cI_k} \NRM{x_v^k - \hat x^k}^2\,.
    \end{align*} 
    Taking the expectation wrt $\cI_k$:
    \begin{align*}
        \esp{- \frac{2\gamma}{n}\sum_{v\in\cI_k}\langle \nabla f_v(x_v^k),\hat x^k-x^\star\rangle}&\leq - \frac{2\gamma}{n}\sum_{v\in\cV} p_v\big( f_v(\hat x^k) -f_v(x^\star)\big)   + \frac{\gamma L}{n}\sum_{v\in\cV} p_v \NRM{x_v^k - \hat x^k}^2\\
        &\leq- \frac{2\gamma n\bar p }{n}\big(f(\hat x^k) -f(x^\star) \big)  + \frac{2\gamma Lp_{\max}}{n} \NRM{\xx^k - \bar \xx^k}^2+ \frac{\gamma L}{n}\sum_{v\in\cV} p_v \NRM{\hat x^k - \bar x^k}^2\\
        &\leq - \frac{2\gamma n\bar p }{n}\big(f(\hat x^k) -f(x^\star) \big)  + \frac{2\gamma Lp_{\max}}{n} \NRM{\xx^k - \bar \xx^k}^2 + 2\gamma L \bar p \NRM{\hat x^k - \bar x^k}^2\,.
    \end{align*}
    Then, for the variance term, we need to bound $\E\NRM{\sum_{v\in\cI_k}\nabla f_v(x_{v}^{k})}^2$.
    For any $(z_v)_{v\in\cV}$, we have $\esp{\NRM{\sum_{v\in\cI_k}z_v}^2}=\esp{\sum_{v,v'\in\cV}\one_{v\in\cV}\one_{v'\in\cV}\langle z_v,z_{v'}\rangle}=\sum_{v\ne v'\in\cV}\one_{v\in\cV}p_vp_{v'}\langle z_v,z_{v'}\rangle+\sum_{v\in\cV}p_v\NRM{z_v}^2\leq \sum_{v\in\cV}p_v\NRM{z_v}^2+\NRM{\sum_{v\in\cV}p_vz_v}^2$. And finally, using convexity of the squared norm, $\NRM{\sum_{v\in\cV}p_vz_v}^2\leq n\bar p\sum_{v\in\cV}p_v\NRM{z_v}^2$.
    Hence, we have \[\E_{\cI_k}\NRM{\sum_{v\in\cI_k}\nabla f_v(x_{v}^{k})}^2\leq  \sum_{v\in\cV} p_v \NRM{\nabla f_v(x_{v}^{k})}^2 +   \NRM{\sum_{v\in\cV} p_v\nabla f_v(x_{v}^{k})}^2\,.\]
    Thus, plugging this in the first inequality,
    \begin{align*}
        \frac{2\gamma n\bar p }{n}\big(\E f(\hat x^k) -f(x^\star) \big) &\leq \esp{\NRM{\hat  \xx^k- \xx^\star}^2-\NRM{\hat  \xx^{k+1}- \xx^\star}^2} + \esp{\frac{\gamma^2\sigma^2\nk + \gamma^2\E\NRM{\sum_{v\in\cI_k}\nabla f_v(x_{v}^{k})}^2}{n^2}}\\
        &\quad + \esp{ \frac{2\gamma Lp_{\max}}{n} \NRM{\xx^k - \bar \xx^k}^2 + 2\gamma L \bar p \NRM{\hat x^k - \bar x^k}^2}\\
        &= \esp{\NRM{\hat  \xx^k- \xx^\star}^2-\NRM{\hat  \xx^{k+1}- \xx^\star}^2} + \frac{\gamma^2\sigma^2 n\bar p + \gamma^2\sum_{v\in\cV}p_v\E\NRM{\nabla f_v(x_{v}^{k})}^2+ \gamma^2\NRM{\sum_{v\in\cV} p_v\nabla f_v(x_{v}^{k})}^2}{n^2}\\
        &\quad + \esp{ \frac{2\gamma Lp_{\max}}{n} \NRM{\xx^k - \bar \xx^k}^2 + 2\gamma L \bar p \NRM{\hat x^k - \bar x^k}^2}\,.
    \end{align*}
    Then, using smoothness, we have that $f(\bar x^k)-f(x^\star)\leq f(\hat x^k)-f(x^\star) + \langle \nabla f(\hat x^k), \hat x^k-\bar x^k\rangle + \frac{L}{2}\NRM{\bar x^k-\hat x^k}\leq 2(f(\hat x^k)-f(x^\star)) +2L\NRM{\bar x^k -\hat x^k}^2$, leading to:
    \begin{align*}
        \frac{2\gamma n\bar p }{n}\big(\E f(\bar x^k) -f(x^\star) \big) &\leq \frac{4\gamma n\bar p }{n}\big(\E f(\hat x^k) -f(x^\star) \big) + \frac{4L\gamma n\bar p }{n}\NRM{\bar x^k -\hat x^k}^2\\
        &\leq 2\esp{\NRM{\hat  \xx^k- \xx^\star}^2-\NRM{\hat  \xx^{k+1}- \xx^\star}^2} + \frac{2\gamma^2\sigma^2 n\bar p + 2\gamma^2\left(\sum_{v\in\cV}p_v\E\NRM{\nabla f_v(x_{v}^{k})}^2+\NRM{\sum_{v\in\cV} p_v\nabla f_v(x_{v}^{k})}^2\right)}{n^2}\\
        &\quad + \esp{ \frac{4\gamma Lp_{\max}}{n} \NRM{\xx^k - \bar \xx^k}^2 + 8\gamma L \bar p  \NRM{\hat x^k - \bar x^k}^2}\,.
    \end{align*}
    We have $\NRM{\hat x^k - \bar x^k}^2\leq \gamma^2B^2$.
    Now, 
    \begin{align*}
        \sum_{v\in\cV}p_v\NRM{\nabla f_v(x_{v}^{k})}^2 & \leq 2\sum_{v\in\cV}p_v\NRM{\nabla f_v(x_{v}^{k})-\nabla f_v(\bar x^{k})}^2+p_v\NRM{\nabla f_v(\bar x^{k})}^2\\
        &\leq 2L^2p_{\max}\NRM{\xx^k-\bar\xx^k}^2+2\sum_{v\in\cV}p_v\NRM{\nabla f_v(\bar x^{k})}^2\\
        &\leq 2L^2p_{max}\NRM{\xx^k-\bar\xx^k}^2+2n\bar p\NRM{\nabla f(\bar x^{k})}^2 + 2n\bar p\zeta^2\,.
    \end{align*}
    Then,
    \begin{align*}
        \NRM{\sum_{v\in\cV}p_v\nabla f_v(x_{v}^{k})}^2 & \leq 2\NRM{\sum_{v\in\cV}p_v\nabla f_v(\bar x^k)}^2 + 2\NRM{\sum_{v\in\cV}p_v(\nabla f_v(x_{v}^{k})-\nabla f_v(\bar x^{k}))}^2\\
        &\leq 2(n\bar p)^2\NRM{\nabla f(\bar x^k)}^2 + 2(n\bar p)\sum_{v\in\cV}p_v\NRM{(\nabla f_v(x_{v}^{k})-\nabla f_v(\bar x^{k}))}^2\\
        &\leq 2(n\bar p)^2\NRM{\nabla f(\bar x^k)}^2 + 2(n\bar p)\sum_{v\in\cV}p_v L^2\NRM{x_{v}^{k}-\bar x^{k}}^2\\
        &\leq 2(n\bar p)^2\NRM{\nabla f(\bar x^k)}^2 + 2(n\bar p)p_{\max}L^2\NRM{\xx^{k}-\bar \xx^{k}}^2\\
    \end{align*}
    Thus, this leads to:
    \begin{align*}
        \frac{2\gamma n\bar p }{n}\big(\E f(\bar x^k) -f(x^\star) \big) &\leq 2\esp{\NRM{\hat  \xx^k- \xx^\star}^2-\NRM{\hat  \xx^{k+1}- \xx^\star}^2} + \frac{2\gamma^2(\sigma^2+2\zeta^2) n\bar p + 8\gamma^2n^2\bar p^2 \NRM{\nabla f(\bar x^k)}^2}{n^2}\\
        &\quad + \esp{ \big(\frac{4\gamma Lp_{\max}}{n} + \frac{2\gamma^2L^2p_{\max}(1+n\bar p) }{n} \big) \NRM{\xx^k - \bar \xx^k}^2 + 8\gamma L \bar p  \NRM{\hat x^k - \bar x^k}^2}\,.
    \end{align*}
    We now use the following lemma.
    \begin{lemma}\label{lem:consensus_advanced}
        For stepsizes $\gamma\leq \frac{\bar\rho}{4L\sqrt{p_{\max}}} $, we have:
        \begin{equation*}
            \sum_{k<K}\esp{\NRM{\xx^k-\bar\xx^k}^2}\leq 4\gamma^2\sigma^2\brhom n\bar p K + 8\gamma^2\bar\rho^{-2}\sum_{v\in\cV}\NRM{\nabla f_v(\bar x^0)}^2 + 16\gamma^2\bar\rho^{-2} n\bar p\sum_{k<K}\big(\NRM{\nabla f(\bar x^k)}^2 + \zeta^2\big)\,.
        \end{equation*}
    \end{lemma}
\begin{proof}[Proof of the lemma]
    Denoting $C_K=\sum_{k<K}\esp{\NRM{\xx^k-\bar\xx^k}^2}$ and using \Cref{lem:consensus_control}, we have
    \begin{align*}
        C_k&\leq 2\gamma^2\sigma^2\brhom \NK + \frac{4\gamma^2}{\bar\rho^2}\sum_{k<K}\esp{\sum_{v\in\cI_k}\NRM{\nabla f_v(x_v^{k-\tau(k,v)})}^2}\\
        &\leq 2\gamma^2\sigma^2\brhom n\bar p K + 8\gamma^2\bar\rho^{-2}\sum_{v\in\cV}\NRM{\nabla f_v(\bar x^0)}^2 + \frac{4\gamma^2}{\bar\rho^2}\sum_{k<K}\sum_{v\in\cV}p_v\esp{\NRM{\nabla f_v(x_v^{k})}^2}\,,
    \end{align*}
    using $\sum_{k<K}\sum_{v\in\cI_k}\NRM{\nabla f_v(x_v^{k-\tau(k,v)})}^2\leq\sum_{k<K}\sum_{v\in\cI_k}\NRM{\nabla f_v(x_v^{k})}^2 + \sum_{v\in\cV}\NRM{\nabla f_v(x_v^0)}^2$.
    Then, $\sum_{v\in\cV}p_v\esp{\NRM{\nabla f_v(x_v^{k})}^2}\leq 2\sum_{v\in\cV}p_v\esp{\NRM{\nabla f_v(\bar x^{k})}^2} +2 \sum_{v\in\cV}p_v\esp{\NRM{\nabla f_v(\bar x^{k})-\nabla f_v(x_v^{k})}^2}\leq 2n\bar p \zeta^2+ 2n\bar p \E\NRM{\nabla f(\bar x^k)}^2 + 2L^2p_{\max}\E\NRM{\xx^k-\bar\xx^k}^2$, which leads to:
    \begin{align*}
        C_K&\leq 2\gamma^2\sigma^2\brhom n\bar p K + 4\gamma^2\bar\rho^{-2}\sum_{v\in\cV}\NRM{\nabla f_v(\bar x^0)}^2 + 8\gamma^2\bar\rho^{-2} n\bar p\sum_{k<K}\big(\NRM{\nabla f(\bar x^k)}^2 + \zeta^2\big)\\
        &\quad + 8\gamma^2L^2p_{\max}\bar\rho^{-2} C_K\,,
    \end{align*}
    leading to the desired result for $\gamma\leq \frac{\bar\rho}{4L\sqrt{p_{\max}}}$.
\end{proof}

Using \Cref{lem:virtual_iterates} and \Cref{lem:consensus_advanced}, we thus have:
\begin{align*}
    \frac{2\gamma n\bar p }{n}\sum_{k<K}\big(\E f(\bar x^k) -f(x^\star) \big) &\leq 2\esp{\NRM{\hat  \xx^0- \xx^\star}^2} + \frac{2\gamma^2(\sigma^2+2\zeta^2) \bar p K}{n} + 4\gamma^2\bar p \sum_{k<K} \esp{\NRM{\nabla f(\bar x^k)}^2}\\
    &\quad + \esp{ \big(\frac{4\gamma Lp_{\max}}{n} + \frac{2\gamma^2L^2p_{\max} }{n} \big) \sum_{k<K}\NRM{\xx^k - \bar \xx^k}^2} + 8\gamma^3 LB^2 \bar p K\\
    &\leq 2\esp{\NRM{\hat  \xx^0- \xx^\star}^2} + \frac{2\gamma^2(\sigma^2+2\zeta^2) \bar p K}{n} + 4\gamma^2\bar p \sum_{k<K} \esp{\NRM{\nabla f(\bar x^k)}^2} + + 8\gamma^3 LB^2 \bar p K\\
    &\quad + \frac{6\gamma Lp_{\max}}{n} \left[ 4\gamma^2\sigma^2\brhom n\bar p K + 8\gamma^2\bar\rho^{-2}\sum_{v\in\cV}\NRM{\nabla f_v(\bar x^0)}^2 + 16\gamma^2\bar\rho^{-2} n\bar p\sum_{k<K}\big(\NRM{\nabla f(\bar x^k)}^2 + \zeta^2\big) \right]\\
    &= 2\esp{\NRM{\hat  \xx^0- \xx^\star}^2} + \big(8\gamma^2L\bar p + 96 \gamma^3L^2p_{\max}\bar p \bar\rho^{-2}\big) \sum_{k<K} \esp{ f(\bar x^k)-f(x^\star)} + \frac{2\gamma^2(\sigma^2+2\zeta^2) \bar p K}{n}\\
    &\quad + \gamma^3 K\left( 8LB^2\bar p + 24L\sigma^2p_{\max}\bar p\brhom + 96 L\zeta^2p_{\max}\bar p\bar\rho^{-2}  \right) + \frac{48\gamma^2 Lp_{\max}\bar\rho^{-2}}{n}\sum_{v\in\cV}\NRM{\nabla f_v(\bar x^0)}^2\,.
\end{align*}
Hence, for stepsizes satisfying $8\gamma L\bar p + 96 \gamma^2L^2p_{\max}\bar p \bar\rho^{-2}\leq \bar p$, which is verified for $\gamma\leq \min\left(\frac{1}{16L}, \frac{\bar \rho}{14L\sqrt{p_{\max}}}  \right)$, we obtain:
\begin{align*}
    \sum_{k<K}\big(\E f(\bar x^k) -f(x^\star) \big) &\leq \frac{2\esp{\NRM{\hat  \xx^0- \xx^\star}^2}}{\gamma \bar p} + \frac{2\gamma(\sigma^2+2\zeta^2) K}{n}+ \gamma^2 K\left( 8LB^2 + 24L\sigma^2p_{\max}\brhom + 96 L\zeta^2p_{\max}\bar\rho^{-2}  \right)\\
    &\quad  + \frac{48\gamma Lp_{\max}\bar\rho^{-2}}{n\bar p}\sum_{v\in\cV}\NRM{\nabla f_v(\bar x^0)}^2\,.
\end{align*}
Optimizing over $\gamma\leq \min\left(\frac{1}{16L}, \frac{\bar \rho}{14L\sqrt{p_{\max}}}, \frac{\bar\rho}{L}  \right)$, this leads to:
\begin{align*}
    \frac{1}{K}\sum_{k<K}\big(\E f(\bar x^k) -f(x^\star) \big) = &\cO\left( \frac{LD^2\left(\frac{1}{\bar p} + \sqrt{\frac{p_{\max}}{\bar p^2}} \brhom  \right)}{K}  + \sqrt{\frac{D^2(\sigma^2+\zeta^2)}{n\bar p K}}  + \left[ \frac{D^2\sqrt{LB^2 + L\sigma^2p_{\max}\bar\rho^{-1} + L\zeta p_{\max}\bar\rho^{-2}}}{\bar p K}  \right]^{\frac{2}{3}}   \right. \\
    &\quad \left. + \frac{\bar\rho^{-1}\frac{p_{\max}}{\bar p}  }{ K } \frac{1}{n} \sum_{v\in\cV} \NRM{\nabla f_v(\bar x^0)}^2 \right) \,.
\end{align*}

\end{proof}

\section{Proof of \Cref{thm:smooth-conv}: smooth-convex case}

\subsection{Homogeneous without sampling}

\begin{proof}
    As before, we have:
    \begin{align*}
        \esp{\NRM{\hat  x^{k+1}- x^\star}^2} \leq \esp{ \NRM{\hat  x^k- x^\star}^2+ T_1+T_2^k+T_3 }  + \frac{\gamma^2\sigma^2\nk + \gamma^2\E\NRM{\sum_{v\in\cI_k}\nabla f_v(x_{v}^{k})}^2}{n^2}\,,
    \end{align*}
    with
    \begin{align*}
        T_1&= - \frac{2\gamma}{n} \sum_{v\in\cI_k} \langle \nabla f_v( x_{v}^k), x_{v}^k- x^\star\rangle \\
        T_2^k&= \frac{2\gamma}{n}\sum_{v\in\cI_k} \langle \nabla f_v( x_{v}^k), x_{v}^k-\bar x^k\rangle \\
        T_3&= \frac{2\gamma}{n}\sum_{v\in\cI_k} \langle \nabla f_v( x_{v}^k),\bar x^k-\hat x^k\rangle \,,
    \end{align*}
    We will bound $T_1,T_2$ as in the proof with the Lipschitz assumption.
    For the term $T_3$, using convexity and \Cref{lem:virtual_iterates}:
    \begin{align*}
        \E T_3 &\leq \frac{\gamma}{n}\Big( C\sum_{v\in\cI_k}\NRM{\nabla f( x_{v}^k)}^2 +\frac{1}{C}\E\NRM{\bar x^k-\hat x^k}^2 \Big)\\
        &\leq \frac{\gamma}{n}\Big( 2LC\sum_{v\in\cI_k}(f( x_{v}^k)-f( x^\star)) +\frac{2\gamma^2}{Cn}\nk(\sigma^2+\sum_{v\in\cV} \NRM{\nabla f(x_v^{k-\tau(v,k)})}^2 )  \Big)\\
        &\leq \frac{2\gamma^2}{Cn^2}\nk\sigma^2 + \frac{\gamma}{n}\Big( 2LC\sum_{v\in\cI_k}(f( x_{v}^k)-f( x^\star)) +\frac{2\gamma^2}{Cn}\nk\sum_{v\in\cV} \NRM{\nabla f(x_v^{k-\tau(v,k)})}^2  \Big)\,.
    \end{align*}
    for $\gamma\leq 1/(nL)$.
    Then, 
    \begin{align*}
        \sum_{k<K}\nk\sum_{v\in\cV} \NRM{\nabla f(x_v^{k-\tau(v,k)})}^2 & \leq \sum_{v\in\cV} \sum_{k<K:v\in\cI_k} \NRM{\nabla f(x_v^{k})}^2\sum_{\ell=k}^{\nex(v,k+1)-1} \nl\\
        &\leq \tau_{\max}\sum_{v\in\cV} \sum_{k<K:v\in\cI_k} \NRM{\nabla f(x_v^{k})}^2\\
        &\leq 2L\tau_{\max}\sum_{v\in\cV} \sum_{k<K:v\in\cI_k} f(x_v^{k})-f(x^\star)\,,
    \end{align*}
    where $\tau_{\max}$ is an upper bound on the maximal compute delay defined as  $\tau_{\max}\geq \sup_{k<K}\sum_{\ell=k}^{\nex(v,k+1)-1} \nl$.

    Thus,
    \begin{align*}
        \frac{2\gamma}{n} \sum_{v\in\cI_k} (\E f( x_{v}^k) -f( x^\star)) &\leq -\esp{\NRM{\hat  x^{k+1}- x^\star}^2} +\esp{ \NRM{\hat  x^k- x^\star}^2} + \frac{\gamma^2\sigma^2\nk}{n^2} + \frac{2\gamma^2 L\nk}{n^2} \sum_{v\in\cI_k} (\E f( x_{v}^k) -f( x^\star)) \\
        & \quad + \frac{2LC\gamma}{n}\sum_{v\in\cI_k}\esp{(f( x_{v}^k)-f( x^\star))} + \frac{\gamma}{Cn}\esp{\NRM{ \xx^k-\bar \xx^k}^2}\\
        &\quad + \frac{2\gamma^2L}{Cn^2}\nk\sigma^2 + \frac{\gamma}{n}\Big( 2LC\sum_{v\in\cI_k}(f( x_{v}^k)-f( x^\star)) +\frac{2\gamma^2}{Cn}\nk\sum_{v\in\cV} \NRM{\nabla f(x_v^{k-\tau(v,k)})}^2  \Big)\,.
    \end{align*}
    Summing over $k<K$, using \Cref{lem:consensus_control} and our bound on $T_3$, we obtain:
    \begin{align*}
        \frac{2\gamma}{n} \sum_{k<K}\sum_{v\in\cI_k} (\E f( x_{v}^k) -f( x^\star)) &\leq \esp{ \NRM{\hat  x^0- x^\star}^2} + \frac{3\gamma^2\sigma^2}{n^2}\NK + \frac{2\gamma L}{n}\big( 2C+\gamma + \frac{2\tau_{\max}\gamma^2}{Cn} \big) \sum_{k<K}\sum_{v\in\cI_k} (\E f( x_{v}^k) -f( x^\star)) \\
        & \quad + \sum_{k<K}\frac{\gamma}{Cn}\esp{\NRM{ \xx^k-\bar \xx^k}^2}\\
        &\leq \esp{ \NRM{\hat  x^0- x^\star}^2} + \frac{\gamma^2\sigma^2}{n^2}\big(1+\frac{2\gamma \bar\rho^{-1}n}{C}\big)\NK+ \sum_{k<K}\frac{\gamma}{Cn}\esp{\NRM{ \xx^k-\bar \xx^k}^2}+\frac{\gamma^3B^2}{Cn}\NK\\
        &\quad + \frac{2\gamma L}{n}\big( 2C+\gamma + \frac{2\tau_{\max}\gamma^2}{Cn}  +  \frac{4\gamma^2}{\bar\rho^{2}} \big) \sum_{k<K}\sum_{v\in\cI_k} (\E f( x_{v}^k) -f( x^\star)) \,,
    \end{align*}
    using \Cref{lem:consensus_advanced} to handle the sum of the terms $\norm{\xx^k - \bar \xx^k}^2$.
    
    Hence, provided that $ 2C+\gamma + \frac{2\tau_{\max}\gamma^2}{Cn}  + \frac{4\gamma^2}{\bar\rho^{2}} \leq \frac{1}{2L}$, which is verified for $C=\frac{1}{8L}$ and $\gamma\leq \frac{1}{4L}\times \frac{1}{1+2\bar\rho^{-1}+4\sqrt{\tau_{\max}/n}}$, we have:
    \begin{align*}
        \frac{\gamma}{n} \sum_{k<K}\sum_{v\in\cI_k} (\E f( x_{v}^k) -f( x^\star)) \leq \esp{ \NRM{\hat  x^0- x^\star}^2} + \frac{\gamma^2\sigma^2}{n^2}\big(3+16L\gamma \bar\rho^{-1}n\big)\NK \,,
    \end{align*}
    leading to, for $\eta=\gamma/n$:
    \begin{align*}
        \esp{f\left(\frac{1}{\NK}\sum_{k=0}^{K-1}\sum_{v\in\cI_k} x_{v}^k\right)-f( x^\star)} \leq \frac{\esp{ \NRM{\hat  x^0- x^\star}^2}}{\eta\NK} + 3\eta \sigma^2 + \eta^2 16L\sigma^2n^2\brhom \,.
    \end{align*}
    Optimizing over $\eta\leq \frac{1}{4L}\times \frac{1}{n(1+2\bar\rho^{-1}) +4 \sqrt{n\tau_{\max}}}$, we thus obtain that:
    \begin{equation*}
        \esp{f\left(\frac{1}{\NK}\sum_{k=0}^{K-1}\sum_{v\in\cI_k} x_{v}^k\right)-f( x^\star)} = \cO \left( \frac{LD^2 (n\brhom+\sqrt{n\tau_{\max})}}{\NK} + \sqrt{\frac{D\sigma^2}{\NK}} + \left[  \frac{D^2\sqrt{L\sigma^2n^2\brhom }}{\NK} \right]^{2/3} \right)\,.
    \end{equation*}
\end{proof}

\subsection{Heterogeneous setting under sampling}\label{sec:heter-smooth}

\begin{proof}
    As in the Lipschitz case, we have:
    \begin{align*}
        \frac{2\gamma n\bar p }{n}\sum_{k<K}\big(\E f(\bar x^k) -f(x^\star) \big) &\leq 2\esp{\NRM{\hat  \xx^0- \xx^\star}^2} + \frac{2\gamma^2(\sigma^2+2\zeta^2) \bar p K}{n} + 4\gamma^2\bar p \sum_{k<K} \esp{\NRM{\nabla f(\bar x^k)}^2}\\
        &\quad + \esp{ \frac{6\gamma Lp_{\max}}{n}  \sum_{k<K}\NRM{\xx^k - \bar \xx^k}^2} + 8\gamma L \bar p \esp{\sum_{k<K} \NRM{\bar x^k-\hat x^k}^2}\,.
    \end{align*}
    Since losses are no longer assumed to be Lipschitz, we cannot bound this last term $\esp{\sum_{k<K} \NRM{\bar x^k-\hat x^k}^2}$ by $\gamma^2 B^2$. However, using \Cref{lem:virtual_iterates},
    \begin{align*}
        \esp{\sum_{k<K} \NRM{\bar x^k-\hat x^k}^2} &\leq \frac{2\gamma^2\sigma^2 K}{n}  + \frac{2\gamma^2}{n}\esp{\sum_{v\in\cV}\sum_{k<K} \NRM{\nabla f_v(x_v^{\prev(v,k)})}^2}\,.
    \end{align*}
    Then, 
    \begin{align*}
        \esp{\sum_{v\in\cV}\sum_{k<K} \NRM{\nabla f_v(x_v^{\prev(v,k)})}^2}&=\sum_{v\in\cV}\sum_{k<K} \esp{\NRM{\nabla f_v(x_v^k)}^2 \one_{v\in\cI_k} (\nex(k,v)-k)}\\
        &=\sum_{v\in\cV}\sum_{k<K} \esp{\NRM{\nabla f_v(x_v^k)}^2 \times \frac{1}{p_v}\times p_v}\\
        &=\sum_{v\in\cV}\sum_{k<K} \esp{\NRM{\nabla f_v(x_v^k)}^2}\\
        &\leq \frac{1}{p_{\min}}\sum_{v\in\cV}\sum_{k<K} p_v\esp{\NRM{\nabla f_v(x_v^k)}^2}\,.
    \end{align*}
    since the random variables $\NRM{\nabla f_v(x_v^k)}^2$, $\one_{v\in\cI_k}$ and $\nex(k,v)-k$ are independent, $\esp{\one_{v\in\cI_k}}=p_v$ (Bernoulli random variable) and $\esp{\nex(k,v)-k}=\frac{1}{p_v}$ (geometric random variable).
    And then, as we proved before, $\sum_{v\in\cV}\sum_{k<K} p_v\esp{\NRM{\nabla f_v(x_v^k)}^2}\leq 2L^2p_{max}\NRM{\xx^k-\bar\xx^k}^2+2n\bar p\NRM{\nabla f(\bar x^{k})}^2 + 2n\bar p\zeta^2$.
    Consequently,
    \begin{align*}
        \frac{2\gamma n\bar p }{n}\sum_{k<K}\big(\E f(\bar x^k) -f(x^\star) \big) &\leq 2\esp{\NRM{\hat  \xx^0- \xx^\star}^2} + \frac{2\gamma^2(\sigma^2+2\zeta^2) \bar p K}{n} + (4\gamma^2\bar p + 32\gamma^3L\bar p \frac{p_{\max}}{p_{\min}})\sum_{k<K} \esp{\NRM{\nabla f(\bar x^k)}^2}\\
        &\quad + \esp{ \big(\frac{6\gamma Lp_{\max}}{n} + \frac{32\gamma^3L^3\bar pp_{\max}}{np_{\min}}\big) \sum_{k<K}\NRM{\xx^k - \bar \xx^k}^2} + \frac{16 \gamma^3\sigma^2 L \bar p K}{n} + \frac{32\gamma^3L\zeta^2\bar p^2}{p_{\min}}\\
        &\leq 2\esp{\NRM{\hat  \xx^0- \xx^\star}^2} + \frac{2\gamma^2(\sigma^2+2\zeta^2) \bar p K}{n} + (4\gamma^2\bar p + 32\gamma^3L\bar p \frac{p_{\max}}{p_{\min}})\sum_{k<K} \esp{\NRM{\nabla f(\bar x^k)}^2}\\
        &\quad + \esp{ \frac{12\gamma Lp_{\max}}{n}  \sum_{k<K}\NRM{\xx^k - \bar \xx^k}^2} + \frac{16 \gamma^3\sigma^2 L \bar p K}{n} + \frac{32\gamma^3L\zeta^2\bar p^2}{p_{\min}}\,.
    \end{align*}
    provided that $\gamma\leq \sqrt{\frac{6p_{\min}}{32L^2p_{\max}}}$.
    Plugging \Cref{lem:consensus_advanced} in here, we obtain:
    \begin{align*}
        \frac{2\gamma n\bar p }{n}\sum_{k<K}\big(\E f(\bar x^k) -f(x^\star) \big) &\leq 2\esp{\NRM{\hat  \xx^0- \xx^\star}^2} + \frac{2\gamma^2(\sigma^2+2\zeta^2) \bar p K}{n} + (4\gamma^2\bar p + 32\gamma^3L\bar p \frac{p_{\max}}{p_{\min}})\sum_{k<K} \esp{\NRM{\nabla f(\bar x^k)}^2}\\
        &\quad  + \frac{16 \gamma^3\sigma^2 L \bar p K}{n} + \frac{32\gamma^3L\zeta^2\bar p^2}{p_{\min}}\\
        &\quad +\frac{12\gamma Lp_{\max}}{n} \left[ 4\gamma^2\sigma^2\brhom n\bar p K + 8\gamma^2\bar\rho^{-2}\sum_{v\in\cV}\NRM{\nabla f_v(\bar x^0)}^2 + 16\gamma^2\bar\rho^{-2} n\bar p\sum_{k<K}\big(\NRM{\nabla f(\bar x^k)}^2 + \zeta^2\big) \right]\\
        &= 2\esp{\NRM{\hat  \xx^0- \xx^\star}^2} + \big(8\gamma^2L\bar p + 192 \gamma^3L^2p_{\max}\bar p \bar\rho^{-2}+ 64\gamma^3L^2\bar p \frac{p_{\max}}{p_{\min}}\big) \sum_{k<K} \esp{ f(\bar x^k)-f(x^\star)}\\
        &\quad + \frac{2\gamma^2(\sigma^2+2\zeta^2) \bar p K}{n}+\gamma^3 K\left( 8LB^2\bar p + 24L\sigma^2p_{\max}\bar p\brhom + 96 L\zeta^2p_{\max}\bar p\bar\rho^{-2}  \right)\\
        &\quad + \frac{96\gamma^3 Lp_{\max}\bar\rho^{-2}}{n}\sum_{v\in\cV}\NRM{\nabla f_v(\bar x^0)}^2\,.    
    \end{align*}
    For $8\gamma^2L\bar p + 192 \gamma^3L^2p_{\max}\bar p \bar\rho^{-2}+ 64\gamma^3L^2\bar p \frac{p_{\max}}{p_{\min}}\leq \gamma\bar p$ which is verified for $\gamma\leq \min\left(\frac{1}{24L},\frac{\bar\rho}{24L\sqrt{p_{\max}}} , \frac{1}{14L\sqrt{\frac{p_{\max}}{p_{\min}}}} \right)$, we have:
    \begin{align*}
        \gamma \bar p \sum_{k<K}\big(\E f(\bar x^k) -f(x^\star) \big) &\leq 2\esp{\NRM{\hat  \xx^0- \xx^\star}^2}+ \frac{96\gamma^3 Lp_{\max}\bar\rho^{-2}}{n}\sum_{v\in\cV}\NRM{\nabla f_v(\bar x^0)}^2\\
        &\quad + \frac{2\gamma^2(\sigma^2+2\zeta^2) \bar p K}{n}+\gamma^3 K\left( 8LB^2\bar p + 24L\sigma^2p_{\max}\bar p\brhom + 96 L\zeta^2p_{\max}\bar p\bar\rho^{-2}  \right)\,,
    \end{align*}
    and thus:
    \begin{align*}
        \frac{1}{K}\sum_{k<K}\big(\E f(\bar x^k) -f(x^\star) \big) &\leq \frac{2\esp{\NRM{\hat  \xx^0- \xx^\star}^2}}{\bar p \gamma K}+ \frac{96\gamma^2 Lp_{\max}\bar\rho^{-2}}{n\bar pK}\sum_{v\in\cV}\NRM{\nabla f_v(\bar x^0)}^2\\
        &\quad + \frac{2\gamma(\sigma^2+2\zeta^2)}{n}+\gamma^2 \left( 8LB^2 + 24L\sigma^2p_{\max}\brhom + 96 L\zeta^2p_{\max}\bar\rho^{-2}  \right)\,.
    \end{align*}
    Now, we use 
    \begin{align*}
        \sum_{v\in\cV}\NRM{\nabla f_v(\bar x^0)}^2 &\leq  \sum_{v\in\cV}\NRM{\nabla f(\bar x^0)}^2 +\zeta^2 \leq  \sum_{v\in\cV}2L (f(x_0)-f(x^\star)) +\zeta^2\,,
    \end{align*}
    so that 
    \begin{align*}
        \frac{96\gamma^2 Lp_{\max}\bar\rho^{-2}}{n\bar pK}\sum_{v\in\cV}\NRM{\nabla f_v(\bar x^0)}^2 &\leq \frac{192\gamma^2 L^2p_{\max}\bar\rho^{-2}}{\bar pK}(f(x_0)-f(x^\star)) + \frac{96\zeta^2\gamma^2 Lp_{\max}\bar\rho^{-2}}{\bar pK}\\
        & \leq \frac{192\gamma^2 L^2p_{\max}\bar\rho^{-2}}{\bar p}\frac{1}{K}\sum_{k<K}\big(\E f(\bar x^k) -f(x^\star) \big) + \frac{96\zeta^2\gamma^2 Lp_{\max}\bar\rho^{-2}}{\bar pK}\\
        &\leq \frac{1}{2}\frac{1}{K}\sum_{k<K}\big(\E f(\bar x^k) -f(x^\star) \big) + 96\zeta^2\gamma^2 Lp_{\max}\bar\rho^{-2}\,,
    \end{align*}
    for $K\geq \frac{1}{\bar p}$ and $\gamma\leq \frac{1\bar\rho}{384L}\sqrt{\frac{\bar p}{p_{\max}}}$.
    Thus, 
        \begin{align*}
        \frac{1}{2K}\sum_{k<K}\big(\E f(\bar x^k) -f(x^\star) \big) &\leq \frac{2\esp{\NRM{\hat  \xx^0- \xx^\star}^2}}{\bar p \gamma K}\\
        &\quad + \frac{2\gamma(\sigma^2+2\zeta^2)}{n}+\gamma^2 \left( 8LB^2 + 24L\sigma^2p_{\max}\brhom + 192 L\zeta^2p_{\max}\bar\rho^{-2}  \right)\,.
    \end{align*}
    Optimizing over admissible $\gamma$'s leads to:
    \begin{align*}
        \frac{1}{K}\sum_{k<K}\big(\E f(\bar x^k) -f(x^\star) \big) = &\cO\left( \frac{LD^2\left(\frac{1}{\bar p} \sqrt{\frac{p_{\max}}{p_{\min}}} + \sqrt{\frac{p_{\max}}{\bar p^2}} \brhom  \right)}{K}  + \sqrt{\frac{D^2(\sigma^2+\zeta^2)}{n\bar p K}}  + \left[ \frac{D^2\sqrt{LB^2 + L\sigma^2p_{\max}\bar\rho^{-1} + L\zeta p_{\max}\bar\rho^{-2}}}{\bar p K}  \right]^{\frac{2}{3}}   \right) \,.
    \end{align*}

\end{proof}

\section{Proof of \Cref{thm:smooth-nonconv}: smooth non-convex case}

\subsection{Homogeneous without sampling}
\begin{proof}
Using $L$-smoothness and a virtual sequence $\hat{x}$ defined in Section~\ref{sec:virtual_sequence}, we have 
\begin{align}\label{eq:descent}
    \E_{k + 1} f(\hat x^{k + 1})
    &\leq f(\hat x^k) - \underbrace{\frac{\gamma}{n} \sum_{v \in \cI_k} \left\langle \nabla f(\hat x^k), \nabla f(x_{v}^{k}) \right\rangle}_{:= T_1} +\frac{L \gamma^2}{2 n^2} \left(\sigma^2\nk + \E\NRM{\sum_{v\in\cI_k}\nabla f(x_{v}^{k})}^2\right)
\end{align}
We separately estimate the middle term as
\begin{align*}
    T_1 &= - \frac{\gamma}{n} \sum_{v \in \cI_k} \left\langle \nabla f(\hat x^k), \nabla f(x_{v}^{k}) \right\rangle = - \frac{\gamma}{n} \sum_{v \in \cI_k} \left\langle \nabla f(\bar x^k), \nabla f(x_{v}^{k}) \right\rangle + \frac{\gamma}{n} \sum_{v \in \cI_k} \left\langle \nabla f(\bar x^k) - \nabla f(\hat x^k), \nabla f(x_{v}^{k}) \right\rangle  \\
    &\leq \frac{\gamma}{n} \sum_{v \in \cI_k} \left(- \frac{1}{2} \norm{\nabla f(\bar x^k)}^2 - \frac{1}{2} \norm{\nabla f(x_v^k)}^2 + \frac{L^2}{2} \norm{x_v^k - \bar x^k}^2 \right) + \frac{\gamma}{n} \sum_{v \in \cI_k} \left( \frac{1}{4} \norm{\nabla f(x_v^k)}^2 + L^2 \norm{\bar x^k - \hat x^k}^2 \right) \\
    & \leq - \frac{\gamma}{4 n} \sum_{v \in \cI_k} \norm{\nabla f(x_v^k)}^2 - \frac{|\cI_k| \gamma}{2n}\norm{\nabla f(\bar x^k)}^2 + \frac{L^2 \gamma}{2 n} \sum_{v \in \cI_k}\norm{x_v^k - \bar x^k}^2 + \frac{\gamma L^2 |\cI_k|}{n} \norm{\bar x^k - \hat x^k}^2
\end{align*}
where we used that for any vectors $a, b \in \R^d$ it holds that $- \langle a, b\rangle = - \frac{1}{2}\norm{a}^2 - \frac{1}{2}\norm{b}^2 + \frac{1}{2}\norm{a - b}^2$ and also it holds that $2 \langle a, b\rangle \leq \gamma \norm{a}^2 + \gamma^{-1}\norm{b}^2$ for any $\gamma > 0$ and we chose $\gamma = 2$. 

We further use Lemma~\ref{lem:virtual_iterates} to estimate the last term
\begin{align*}
    T_1 &\leq - \frac{\gamma}{4 n} \sum_{v \in \cI_k} \norm{\nabla f(x_v^k)}^2 - \frac{|\cI_k| \gamma}{2n}\norm{\nabla f(\bar x^k)}^2 + \frac{L^2 \gamma}{2 n} \norm{\xx^k-\bar \xx^k}^2 + \frac{2L^2\gamma^3 \nk }{n^2}\left( \sigma^2 + \sum_{v\in\cV} \esp{\NRM{\nabla f_v(x_v^{\prev(v,k)})}^2}\right)
\end{align*}
Putting this estimate of $T_1$ back into \eqref{eq:descent} we get
\begin{align*}
    \E_{k + 1} f(\hat x^{k + 1})
    &\leq f(\hat x^k) +\frac{L \gamma^2\sigma^2\nk}{2 n^2} + \frac{L \gamma^2}{2 n} \sum_{v\in\cI_k}\E\NRM{\nabla f(x_{v}^{k})}^2 - \frac{\gamma}{4 n} \sum_{v \in \cI_k} \norm{\nabla f(x_v^k)}^2 - \frac{|\cI_k| \gamma}{2n}\norm{\nabla f(\bar x^k)}^2 \\
    & \qquad\qquad + \frac{L^2 \gamma}{2 n} \norm{\xx^k-\bar \xx^k}^2 + \frac{2L^2\gamma^3 \nk }{n^2}\left(\sigma^2 + \sum_{v\in\cV} \esp{\NRM{\nabla f_v(x_v^{\prev(v,k)})}^2}\right)
\end{align*}
Using that $\gamma < \frac{1}{4 L}$ we estimate
\begin{align*}
    \E_{k + 1} f(\hat x^{k + 1})
    &\leq f(\hat x^k) - \frac{\gamma}{8 n} \sum_{v \in \cI_k} \norm{\nabla f(x_v^k)}^2 - \frac{|\cI_k| \gamma}{2n}\norm{\nabla f(\bar x^k)}^2 + \frac{L^2 \gamma}{2 n} \norm{\xx^k-\bar \xx^k}^2\\
    & \qquad\qquad  + \frac{2L^2\gamma^3 \nk}{n^2}\left(\sigma^2 + \sum_{v\in\cV} \esp{\NRM{\nabla f_v(x_v^{\prev(v,k)})}^2}\right) + \frac{L \gamma^2\sigma^2\nk}{2 n^2} 
\end{align*}
Taking the full expectation and summing over all the iterations $k$, we get
\begin{align*}
    \sum_{k < K} \frac{\nk \gamma}{2n} \E \norm{\nabla f(\bar x^k)}^2 &\leq (f(x^0) - f^\star) - \frac{\gamma}{8 n} \sum_{k < K} \sum_{v \in \cI_k} \E \norm{\nabla f(x_v^k)}^2 + \frac{L^2 \gamma}{2n} \sum_{k < K} \E \norm{\xx^k - \bar \xx^k}^2 \\
    & \qquad \qquad + \frac{L \gamma^2\sigma^2\sum_{k < K}\nk}{2 n^2} (1 + 4 L \gamma) +  \frac{2L^2\gamma^3 }{n^2} \sum_{k < K} \sum_{v\in\cV}\nk\esp{\NRM{\nabla f_v(x_v^{\prev(v,k)})}^2}
\end{align*}
For the third term we use Lemma~\ref{lem:consensus_control}, and for the last term we use that 
    \begin{align*}
        \sum_{k<K}\nk\sum_{v\in\cV} \NRM{\nabla f(x_v^{\prev(v,k)})}^2 & \leq \sum_{v\in\cV} \sum_{k<K:v\in\cI_k} \NRM{\nabla f(x_v^{k})}^2\sum_{\ell=k}^{\nex(v,k+1)-1} \nl\\
        &\leq \tau_{\max}\sum_{v\in\cV} \sum_{k<K:v\in\cI_k} \NRM{\nabla f(x_v^{k})}^2
    \end{align*}
where $\tau_{\max}$ is an upper bound on the maximal compute delay defined as  $\tau_{\max}\geq \sup_{k<K}\sum_{\ell=k}^{\nex(v,k+1)-1} \nl$.
For estimating the third term with Lemma 2, we also use that
\begin{align*}
    \sum_{k<K}\sum_{v\in\cI_k}\esp{\NRM{\nabla f_v(x_v^{k-\tau(k,v)})}^2} &\leq \sum_{k<K}\sum_{v\in\cI_k}\esp{\NRM{\nabla f_v(x_v^{k})}^2}
\end{align*}

We therefore get
\begin{align*}
    \sum_{k < K} \frac{\nk \gamma}{2n} \E\norm{\nabla f(\bar x^k)}^2 &\leq (f(x^0) - f^\star) - \frac{\gamma}{8 n} \sum_{k < K} \sum_{v \in \cI_k} \E \norm{\nabla f(x_v^k)}^2 + \frac{L^2 \gamma}{2n} 2\gamma^2\sigma^2\brhom \NK \\
    & \qquad \qquad + \frac{2 L^2 \gamma^3}{n \bar\rho^2} \sum_{k<K}\sum_{v\in\cI_k}\esp{\NRM{\nabla f(x_v^{k})}^2} \\
    & \qquad \qquad + \frac{L \gamma^2\sigma^2\sum_{k < K}\nk}{2 n^2} (1 + 4 L \gamma) +  \frac{2L^2\gamma^3 }{n^2} \tau_{\max} \sum_{k < K} \sum_{v \in \cI_k} \E \norm{\nabla f(x_v^k)}^2
\end{align*}
We further use that the stepsize $\gamma < \frac{1}{8L} (\sqrt{\frac{n}{\tau_{\max}}} + \bar\rho)$
\begin{align*}
    \sum_{k < K} \frac{\nk \gamma}{2n} \E\norm{\nabla f(\bar x^k)}^2 &\leq (f(x^0) - f^\star) + \frac{L^2}{n}\gamma^3\sigma^2\brhom \NK + \frac{L \gamma^2\sigma^2\sum_{k < K}\nk}{n^2}
\end{align*}
Therefore,
\begin{align*}
    \sum_{k < K} \nk \E\norm{\nabla f(\bar x^k)}^2 &\leq \frac{2n}{\gamma}(f(x^0) - f^\star) + 2 L^2\gamma^2\sigma^2\brhom \NK + \frac{2 L \gamma\sigma^2\sum_{k < K}\nk}{n}
\end{align*}
Denoting $T = \NK$ and tuning over the stepsize $\gamma$, we get 
\begin{align*}
    \frac{1}{\NK}\sum_{k < K} \nk \E\norm{\nabla f(\bar x^k)}^2 \leq \frac{16 L F_0(\sqrt{n \tau_{\max}} + n \bar\rho^{-1})}{T} + 4 \left(\frac{L \sigma^2 F_0}{T} \right)^{\frac{1}{2}} + 4 \left(\frac{L \sigma n F_0}{T \sqrt{\bar \rho}}\right)^{\frac{2}{3}}
\end{align*}
where $F_0 = (f(x^0) - f^\star)$.
\end{proof}

\subsection{Heterogeneous with sampling}
\begin{proof}
    Using $L$-smoothness of $f$,
    \begin{align}\label{eq:initial}
        \E_{k + 1} f(\hat x^{k + 1})
        &\leq f(\hat x^k) - \underbrace{\frac{\gamma}{n} \E\sum_{v \in \cI_k} \left\langle \nabla f(\hat x^k), \nabla f_v(x_{v}^{k}) \right\rangle}_{:= T_1} +\frac{L \gamma^2}{2 n^2} \left(\sigma^2 \E \nk + \E\NRM{\sum_{v\in\cI_k}\nabla f_v(x_{v}^{k})}^2\right)
    \end{align}
    We separately estimate the $T_1$ term
    \begin{align*}
        T_1 &= - \frac{\gamma}{n}\E \sum_{v \in \cI_k} \left\langle \nabla f(\hat x^k), \nabla f_v(x_{v}^{k}) \right\rangle = - \frac{\gamma}{n}\E\sum_{v \in \cI_k} \left\langle \nabla f(\hat x^k), \nabla f_v(\bar x^{k}) \right\rangle + \frac{\gamma}{n}\sum_{v \in \cI_k} \E \left\langle \nabla f(\hat x^k), \nabla f_v(\bar x^{k}) - \nabla f_v(x_{v}^{k})\right\rangle \\
        &=  - \gamma \bar p \left\langle \nabla f(\hat x^k), \nabla f(\bar x^{k}) \right\rangle + \frac{\gamma}{n}\sum_{v \in \cI_k} \E \left\langle \nabla f(\hat x^k), \nabla f_v(\bar x^{k}) - \nabla f_v(x_{v}^{k})\right\rangle \\
        & \leq \gamma \bar p \left(-\frac{1}{2}\norm{\nabla f(\hat x^k)}^2 - \frac{1}{2}\norm{\nabla f(\bar x^{k})}^2 + \frac{L^2}{2} \norm{\hat x^k - \bar x^{k}}^2 \right) + \frac{\gamma}{n}\left(\frac{1}{2} n \bar p \norm{\nabla f(\hat x^k)}^2 + \frac{L^2}{2} \E \sum_{v \in \cI_k}\norm{x_v^k - \bar x^k}^2\right)
    \end{align*}
    Since $\E \sum_{v \in \cI_k}\nabla f_v (\bar x^k) = n \bar p~ \nabla f(\bar x^k)$, and $\E |\cI_k| = n \bar p$. We further use that $\E \sum_{v \in \cI_k}\norm{x_v^k - \bar x^k}^2 = \sum_{v \in \cV}p_v \norm{x_v^k - \bar x^k}^2 \leq p_{\max} \norm{\xx^k - \bar\xx^k}^2$.
    Therefore,
    \begin{align*}
        T_1 \leq - \frac{\gamma \bar p}{2} \norm{\nabla f(\bar x^k)} + \frac{\gamma \bar p L^2}{2}\norm{\hat x^k - \bar x^{k}}^2 + \frac{\gamma L^2 p_{\max}}{2 n} \norm{\xx^k - \bar\xx^k}^2
    \end{align*}
    Putting this back to \eqref{eq:initial} and summing it up over $K$, we get
    \begin{align*}
        \frac{\gamma \bar p}{2} \sum_{k < K} \E \norm{\nabla f(\bar x^k)}^2 &\leq (f(x^0) - f^\star) + \frac{\gamma \bar p L^2}{2} \sum_{k < K} \E \norm{\hat x^k - \bar x^k}^2 + \frac{\gamma L^2 p_{\max}}{2n }\sum_{k < K}\E \norm{\xx^k - \bar\xx^k}^2 + \frac{L \gamma^2 \sigma^2 \bar p K}{2 n} \\
        &+ \frac{L \gamma^2 }{2 n^2} \sum_{k < K} \E\NRM{\sum_{v\in\cI_k}\nabla f_v(x_{v}^{k})}^2
    \end{align*}
We use calculations from Section~\ref{sec:het-lipshitz-convex} to further estimate the last term
\begin{align*}
    \E\NRM{\sum_{v\in\cI_k}\nabla f_v(x_{v}^{k})}^2 \leq  \sum_{v\in\cV} p_v \NRM{\nabla f_v(x_{v}^{k})}^2 +   \NRM{\sum_{v\in\cV} p_v\nabla f_v(x_{v}^{k})}^2
\end{align*}
\begin{align}\label{eq:sum_pv_norm}
        \sum_{v\in\cV}p_v\NRM{\nabla f_v(x_{v}^{k})}^2 &\leq 2L^2p_{max}\NRM{\xx^k-\bar\xx^k}^2+2n\bar p\NRM{\nabla f(\bar x^{k})}^2 + 2n\bar p\zeta^2\,.
    \end{align}
    \begin{align*}
        \NRM{\sum_{v\in\cV}p_v\nabla f_v(x_{v}^{k})}^2 &\leq 2(n\bar p)^2\NRM{\nabla f(\bar x^k)}^2 + 2(n\bar p)p_{\max}L^2\NRM{\xx^{k}-\bar \xx^{k}}^2\\
    \end{align*}
    We therefore get
    \begin{align*}
        \frac{\gamma \bar p}{2} \sum_{k < K} \E \norm{\nabla f(\bar x^k)}^2 &\leq (f(x^0) - f^\star) + \frac{\gamma \bar p L^2}{2} \sum_{k < K} \E \norm{\hat x^k - \bar x^k}^2 + \frac{\gamma L^2 p_{\max}}{2n }\sum_{k < K}\E \norm{\xx^k - \bar\xx^k}^2 + \frac{L \gamma^2 \sigma^2 \bar p K}{2 n} \\
        &\qquad\qquad + \frac{L \gamma^2}{n^2} \left( (L^2p_{\max} + (n\bar p)p_{\max}L^2)\NRM{\xx^k-\bar\xx^k}^2 + (n\bar p + (n \bar p)^2)\NRM{\nabla f(\bar x^{k})}^2 + n\bar p\zeta^2\right)\\
        &\leq (f(x^0) - f^\star) + \frac{\gamma \bar p L^2}{2} \sum_{k < K} \E \norm{\hat x^k - \bar x^k}^2 + \frac{\gamma L^2 p_{\max}}{2n}\left[1 + 2 L \gamma \left(\frac{1}{n} + \bar p\right)\right]\sum_{k < K}\E \norm{\xx^k - \bar\xx^k}^2 \\
        &\qquad\qquad + \frac{L \gamma^2 \bar p K \sigma^2}{2 n} + \frac{L \gamma^2 n\bar p (1+ n \bar p)}{n^2} \sum_{k < K} \NRM{\nabla f(\bar x^{k})}^2 + \frac{L \gamma^2 \bar p\zeta^2 K}{n}\\
    \end{align*}
    We further use Lemma~\ref{lem:virtual_iterates} to estimate the term with $\E \norm{\hat x^k - \bar x^k}^2$:
    \begin{align*}
        \esp{\NRM{\hat x^k-\bar x^k}^2}\leq \frac{2\gamma^2}{n}\left( \sigma^2 + \sum_{v\in\cV} \esp{\NRM{\nabla f_v(x_v^{\prev(v,k)})}^2}\right)
    \end{align*}
    And we use calculations from Section~\ref{sec:heter-smooth} estimating
    \begin{align*}
        \esp{\sum_{v\in\cV}\sum_{k<K} \NRM{\nabla f_v(x_v^{\prev(v,k)})}^2} &\leq \frac{1}{p_{\min}}\sum_{v\in\cV}\sum_{k<K} p_v\esp{\NRM{\nabla f_v(x_v^k)}^2}\,.
    \end{align*}
    And \eqref{eq:sum_pv_norm} to estimate the last term. Therefore we get
    \begin{align*}
        \sum_{k < K}\esp{\NRM{\hat x^k-\bar x^k}^2}\leq \frac{2\gamma^2}{n}\left( \sigma^2 K + \frac{1}{p_{\min}}\sum_{k<K} 2L^2p_{max}\NRM{\xx^k-\bar\xx^k}^2+2n\bar p\sum_{k<K}\NRM{\nabla f(\bar x^{k})}^2 + 2n\bar p\zeta^2 K\right)
    \end{align*}
    And
    \begin{align*}
        \frac{\gamma \bar p}{2} \sum_{k < K} \E \norm{\nabla f(\bar x^k)}^2 &\leq (f(x^0) - f^\star) + \frac{\gamma L^2 p_{\max}}{2n}\left[1 + 2 L \gamma \left(\frac{1}{n} + \bar p\right) + \frac{4 \gamma^2 L^2 \bar p}{p_{\min}}\right]\sum_{k < K}\E \norm{\xx^k - \bar\xx^k}^2 \\
        & + \frac{L \gamma^2 \bar p K \sigma^2}{2 n} (1 + \gamma L) + \frac{L \gamma^2 n\bar p (1+ n \bar p + 2 \gamma L n \bar p/p_{\min})}{n^2} \sum_{k < K} \NRM{\nabla f(\bar x^{k})}^2 + \frac{L \gamma^2 \bar p\zeta^2 K}{n}(1 + 2 n \bar p \gamma L)\\
    \end{align*}
We further use that $\gamma < \min\left\{\frac{1}{4L}, \frac{\sqrt{\bar p}}{4 L \sqrt{p_{\min}}} \right\}$
    \begin{align*}
        \frac{\gamma \bar p}{2} \sum_{k < K} \E \norm{\nabla f(\bar x^k)}^2 &\leq (f(x^0) - f^\star) + \frac{\gamma L^2 p_{\max}}{n}\sum_{k < K}\E \norm{\xx^k - \bar\xx^k}^2 \\
        & + \frac{L \gamma^2 \bar p K \sigma^2}{2 n} (1 + \gamma L) + 3 L \gamma^2 \bar p \sum_{k < K} \NRM{\nabla f(\bar x^{k})}^2 + \frac{L \gamma^2 \bar p\zeta^2 K}{n}(1 + 2 n \bar p \gamma L)\\
    \end{align*}
We further use Lemma~\ref{lem:consensus_advanced}:
    \begin{align*}
        \frac{\gamma \bar p}{2} \sum_{k < K} \E \norm{\nabla f(\bar x^k)}^2 &\leq  \frac{\gamma L^2 p_{\max}}{n} \left[ 4\gamma^2\sigma^2\brhom n\bar p K + 8\gamma^2\bar\rho^{-2}\sum_{v\in\cV}\NRM{\nabla f_v(\bar x^0)}^2 + 16\gamma^2\bar\rho^{-2} n\bar p\sum_{k<K}\big(\NRM{\nabla f(\bar x^k)}^2 + \zeta^2\big) \right] \\
        & + \frac{L \gamma^2 \bar p K \sigma^2}{2 n} (1 + \gamma L) + 3 L \gamma^2 \bar p \sum_{k < K} \NRM{\nabla f(\bar x^{k})}^2 + \frac{L \gamma^2 \bar p\zeta^2 K}{n}(1 + 2 n \bar p \gamma L) + (f(x^0) - f^\star)\\
        &\leq (f(x^0) - f^\star) + 4\gamma^3 L^2 p_{\max} \bar p \bar \rho^{-1} K \sigma^2 + \frac{8 \gamma^3 L^2 p_{\max} \bar \rho^{-2} }{n}\sum_{v\in\cV}\NRM{\nabla f_v(\bar x^0)}^2 + \frac{L \gamma^2 \bar p K \sigma^2}{2 n} (1 + \gamma L)\\
        & + L \gamma^2\bar p (3 + 16 \gamma L \bar \rho^{-2} p_{\max}) \sum_{k < K} \NRM{\nabla f(\bar x^{k})}^2 + \frac{2 L \gamma^2 \bar p\zeta^2 K}{n}(1 + 8 \rho^{-2} n \gamma L p_{\max})
    \end{align*}
    Taking the stepsize $\gamma < \min\{\frac{1}{24 L}, \frac{\bar \rho}{16 L \sqrt{p_{\max}}}\}$ we get:
    \begin{align*}
        \frac{\gamma \bar p}{4} \sum_{k < K} \E \norm{\nabla f(\bar x^k)}^2 &\leq (f(x^0) - f^\star) + 4\gamma^3 L^2 p_{\max} \bar p \bar \rho^{-1} K \sigma^2 + \frac{8 \gamma^3 L^2 p_{\max} \bar \rho^{-2} }{n}\sum_{v\in\cV}\NRM{\nabla f_v(\bar x^0)}^2 + \frac{2L \gamma^2 \bar p K \sigma^2}{2 n}\\
        & + \frac{2 L \gamma^2 \bar p\zeta^2 K}{n}(1 + 8 \rho^{-2} n \gamma L p_{\max})
    \end{align*}
We conclude as in the smooth convex case by tuning the stepsize and getting rid of the $\frac{8 \gamma^3 L^2 p_{\max} \bar \rho^{-2} }{n}\sum_{v\in\cV}\NRM{\nabla f_v(\bar x^0)}^2$.
\end{proof}

\end{document}